\documentclass{amsproc}

\usepackage{array,amscd,amsmath,amsthm}
\usepackage{amssymb}
\usepackage[all]{xy}

\usepackage{psfrag}

\usepackage{graphicx}

\makeatletter
\newenvironment{figurehere}
  {\def\@captype{figure}}
  {}
\makeatother

\newtheorem{theorem}{Theorem}[section]
\newtheorem{lemma}[theorem]{Lemma}
\newtheorem{corollary}[theorem]{Corollary}
\newtheorem{proposition}[theorem]{Proposition}

\theoremstyle{definition}
\newtheorem{definition}[theorem]{Definition}
\newtheorem{example}[theorem]{Example}
\newtheorem{xca}[theorem]{Exercise}

\theoremstyle{remark}
\newtheorem{remark}[theorem]{Remark}

\numberwithin{equation}{section}

\newcommand{\abs}[1]{\lvert#1\rvert}

\newcommand{\blankbox}[2]{%
  \parbox{\columnwidth}{\centering
    \setlength{\fboxsep}{0pt}%
    \fbox{\raisebox{0pt}[#2]{\hspace{#1}}}%
  }%
}

\newcommand{\C}{\mathbb{C}}
\newcommand{\Z}{\mathbb{Z}}
\newcommand{\Q}{\mathbb{Q}}
\newcommand{\R}{\mathbb{R}}
\newcommand{\N}{\mathbb{N}}

\def\Hol{\mathrm{Hol}}
\def\hol{\mathrm{hol}}

\def\tr{\mathrm{Tr}}

\def\H{\mathbb{H}}

\def\Teich{\mathcal{T}}

\def\pa{\partial}
\def\Pr{\mathbb{P}}

\def\k{\kappa}
\def\rar{\rightarrow}

\def\hra{\hookrightarrow}
\def\a{\alpha}
\def\Af{\mathfrak{A}}

\def\b{\beta}
\def\Ao{\Af^\circ}
\def\A8{\Af_{\infty}}
\def\Yy{\mathcal{Y}}

\def\La{\Lambda}
\def\i{\iota}
\def\g{\gamma}
\def\Mod{\mathrm{Mod}}
\def\Bb{\mathfrak{B}}

\def\d{\delta}
\def\lra{\longrightarrow}

\def\s{\sigma}

\def\e{\varepsilon}

\def\ol#1{\overline{#1}}
\def\wh#1{\widehat{#1}}

\def\ti#1{\tilde{#1}}
\def\wti#1{\widetilde{#1}}
\def\Rep{\mathcal{R}}

\def\psl{\mathfrak{sl}}

\def\Pr{\mathbb{P}}
\def\Qq{\mathcal{Q}}
\def\ora#1{\overrightarrow{#1}}
\def\ola#1{\overleftarrow{#1}}
\def\bold#1{\mbox{\boldmath$#1$}}
\def\ua{\bold{\a}}
\def\ub{\bold{\b}}

\def\th{\vartheta}

\def\dev{\mathrm{dev}}

\def\Pp{\mathcal{P}}

\def\Sch{\text{\boldmath$S$}}
\def\gr8{\mathrm{gr}_{\infty}}

\def\Met{\mathfrak{Met}}
\def\AMet{\mathfrak{AMet}}
\def\Diff{\mathrm{Diff}}
\def\Dd{\mathcal{D}}
\def\Dhat{\widehat{\Dd}}

\def\Conf{\mathfrak{Conf}}
\def\Yhat{\widehat{\Yy}}

\def\PGL{\mathrm{PGL}}
\def\PU{\mathrm{PU}}
\def\PSL{\mathrm{PSL}}
\def\P{\mathrm{P}}
\def\PGL{\mathrm{PGL}}
\def\qu{\backslash}
\def\SE{\mathrm{SE}}
\def\Arc{\mathfrak{Arc}}
\def\TA{\mathfrak{TA}}
\def\Bl{\mathrm{Bl}}
\def\val{\mathrm{val}}
\def\ad{\mathrm{ad}}
\def\Ad{\mathrm{Ad}}
\def\psl{\mathfrak{sl}}
\def\lieg{\mathfrak{g}}
\def\gru{\Gamma}

\begin{document}

\title{Poisson structures on the Teichm\"uller space
of hyperbolic surfaces with conical points}

\author{Gabriele Mondello}
\address{Department of Mathematics, Imperial College of London,
South Kensington Campus, London SW7 2AZ, UK}
\email{g.mondello@imperial.ac.uk}


\subjclass{Primary 53D30, 30F60; Secondary 32G15}
\date{November 30, 2008.}


\keywords{Hyperbolic surface, moduli space, representation variety,
Poisson structure}

\begin{abstract}
In this paper two Poisson structures on the moduli space of hyperbolic
surfaces with conical points are compared: the Weil-Petersson one
and the $\eta$ coming from the representation variety. We show that
they are multiple of each other, if the angles do not exceed $2\pi$.
Moreover, we exhibit an explicit formula for $\eta$ in terms
of hyperbolic lengths of a suitable system of arcs.
\end{abstract}

\maketitle

\section{Introduction}
The uniformization theorem for hyperbolic surfaces of genus $g$
with conical points
(\cite{mcowen:point-singularity},
\cite{mcowen:prescribed} and \cite{troyanov:prescribing};
see Section~\ref{sec:admissible})
allows to identify the space $\Yy(S,x)(\th)$ of hyperbolic metrics
on $S$ (up to isotopy) with angles $\th=(\th_1,\dots,\th_n)$
at the marked points $x=(x_1,\dots,x_n)$
to the Teichm\"uller space $\Teich(S,x)$
(see Section~\ref{sec:spaces}).

It is thus possible
to define a Weil-Petersson pairing $h^*_{WP,\th}=
g^*_{WP,\th}+i\eta_{WP,\th}$ on the cotangent space of
$\Teich(S,x)$ at $J$ as
\[
h^*_{WP,\th}(\varphi,\psi):=-\frac{1}{4}\int_S
g^{-1}_{\th}(\varphi,\ol{\psi})
\]
where $\varphi,\psi \in H^0(S,K_S^{\otimes 2}(x))
\cong T^*\Teich(S,x)$ are holomorphic with respect to $J$
and $g_{\th}$ is the area form
of the unique hyperbolic metric conformally equivalent
to $J$ and with angles $\th$.
In particular, $h^*_{WP,0}$ is the standard Weil-Petersson
dual Hermitian form.

As the angles $\th_j$ become larger (but still
satisfy the hyperbolicity constraint
$(2g-2+n)\pi>\th_1+\dots+\th_n$),
the situation ``deteriorates''.
In particular, if some $\th_k\geq \pi$,
no collar lemma for the conical points holds
(see Lemma~\ref{lemma:collar}).
Moreover, for some choice of the hyperbolic metric $g$ on $S$,
there can be no smooth geodesic $\hat{\g}\subset S\setminus x$
isotopic to a given loop $\g$ in $S\setminus x$.

As noticed in \cite{schumacher-trapani:cone}, $g_{WP,\th}$ becomes
smaller as $\th$ increases.
Moreover, as $\th_k$ approaches $2\pi$ from below,
the fibers of the forgetful map
$f_k:\Teich(S,x)\lra \Teich(S,x\setminus\{x_k\})$
(metrically) shrink and 
$h_{WP,\th}$ converges to $f_k^*(h_{WP,\th_{\hat{k}}})$, where
$\th_{\hat{k}}=(\th_1,\dots,\hat{\th}_k,\dots,\th_n)$.

So, for $\th\in[0,2\pi)^n$ the pairing $h_{WP,\th}$ defines
a K\"ahler metric \cite{schumacher-trapani:variation},
but it gets more and more degenerate whenever some $\th_k$
overcomes the ``walls'' $2\pi\N_+$.

From a different point of view, there is another interesting way to define
an alternate pairing on $\Teich(S,x)$. In fact, a choice of
$\th$ (such that no $\th_j$ is a positive multiple of $2\pi$)
permits to real-analytically
identify $\Teich(S,x)$ to the space of
Poincar\'e projective structures
(defined by requiring the developing map to be a local isometry)
inside the space of
all ``moderately singular'' projective structures $\Pp(S,x)$
(see Section~\ref{sec:projective}).
Moreover, an important
theorem of Luo \cite{luo:holonomy}
(which we reprove in a different way)
asserts that, if $\th_k\notin 2\pi\N_+$ for all $1\leq k\leq n$, then
the holonomy map $\Pp(S,x)\lra \Rep(\pi_1(S\setminus x),\PSL_2(\C))
=\mathrm{Hom}(\pi_1(S\setminus x),\PSL_2(\C))/\PSL_2(\C)$
is a real-analytic local diffeomorphism.

Our first results, described more extensively in
Theorem~\ref{thm:holonomy}, Proposition~\ref{prop:injectivity}
and Proposition~\ref{prop:integral},
can be summarized in the following.

\begin{theorem}
Let $\La_-:=\{\th\in \R_{\geq 0}^n\,|\,
\th_1+\dots+\th_n<2\pi(2g-2+n)\}$
and $\La_-^\circ:=\La_-\cap (\R_{\geq 0}\setminus 2\pi\N_+)^n$. Then:
\begin{itemize}
\item[(a)]
the holonomy map $\Teich(S,x)\times \La_-^\circ
\cong\Yy(S,x)(\La_-^\circ)\lra \Rep(\pi_1(S\setminus x),\PSL_2(\R))$
is a real-analytic local diffeomorphism;
\item[(b)]
the restriction of the holonomy map
to $\{\th\in \La_-\,|\, \th_j\leq \pi\ \forall j\}$
is injective;
\item[(c)]
if $\th_i,\th_j>\pi$ (for $i\neq j$), then
the holonomy map $\Teich(S,x)\cong\Yy(S,x)(\th)\lra
\Rep(\pi_1(S\setminus x),\PSL_2(\R))$ is not injective.
\end{itemize}
\end{theorem}

The local behavior around $g$ of the holonomy map can be studied
using special coordinates (the $a$-lengths), namely the hyperbolic lengths
of a maximal system of arcs $\ua$ (which are simple, non-homotopic,
non-intersecting unoriented paths between pairs of points in $x$)
adapted to $g$ (see Section~\ref{sec:arcs}).
Actually, if the angles are smaller than $\pi$,
the $a$-lengths allow to reconstruct
the full geometry of the surface, so that we can obtain also the injectivity.
The existence of adapted triangulations is not obvious if the angles are not
small and it is a consequence of the Voronoi decomposition of $(S,x)$
(see Section~\ref{sec:voronoi}).
We remark that, as $\th\rar 0$, the Voronoi decomposition and the associated
(reduced) $a$-lengths extend to the space of decorated hyperbolic
surfaces with cusps (see Section~\ref{sec:decorated}),
thus recovering Penner's lambda lengths \cite{penner:decorated}.

Back to the previous alternate pairings,
the representation space $\Rep(\pi_1(S\setminus x),\PSL_2(\R))$
is naturally endowed with a Poisson structure $\eta$ at its smooth points
induced by the Lefschetz duality on $(S,x)$ and a $\PSL_2(\R)$-invariant
nondegenerate symmetric bilinear product on $\psl_2(\R)$
(see Section~\ref{sec:poisson}).

Thus,
we can compare $\eta_{WP,\th}$ with the pull-back of $\eta$
via the holonomy map,
whenever the angles do not belong to $2\pi\N$. Adapting the work of
Goldman \cite{goldman:symplectic},
we prove that the Shimura isomorphism
holds for angles smaller than $2\pi$.

\begin{theorem}
If $\th\in \La_-\cap(0,2\pi)^n$, then
\[
\eta_{WP,\th}=\frac{1}{8} \eta\Big|_{\th}
\]
as dual symplectic forms on $\Yy(S,x)(\th)\cong\Teich(S,x)$.
\end{theorem}

Clearly, we could not ask the equality to hold for larger angles
$\th\in\La_-^\circ$,
as $\eta_{WP,\th}$ becomes degenerate, whilst $\eta\Big|_{\th}$ is
not. However, in proving the theorem we obtain the following.

\begin{corollary}
If $\th\in\La_-$, then
\[
\eta_{WP,\th}(\varphi,\psi)=\frac{1}{8}\eta\Big|_{\th}(\varphi,\psi)
\]
for $\varphi,\psi\in T^*\Teich(S,x)$ whenever
both hand-sides converge
(the right-hand side is always finite if $\th_j\notin 2\pi\N_+$
for all $j$).
\end{corollary}

Finally, in Section~\ref{sec:explicit}
we find an explicit formula for $\eta$ in terms of
the $a$-length coordinates.

\begin{theorem}
Let $\ua$ be a triangulation of $(S,x)$ adapted to
$g\in\Yy(S,x)(\La_-^\circ)$ and let $a_k=\ell_{\a_k}$.
Then the Poisson structure $\eta$ at $g$ can be expressed
in term of the $a$-lengths as follows
\[
\eta_g=
\sum_{h=1}^n
\sum_{\substack{\mathrm{s}(\ora{\a_i})=x_h \\ 
\mathrm{s}(\ora{\a_j})=x_h}}
\frac{\sin(\th_h/2-d(\ora{\a_i},\ora{\a_j}))}{\sin(\th_h/2)}
\frac{\pa}{\pa a_i}\wedge
\frac{\pa}{\pa a_j}
\]
where $s(\ora{\a}_k)$ is the starting point of the oriented
arc $\ora{\a}_k$ and $d(\ora{\a_i},\ora{\a_j})$ is the angle
spanned by rotating the tangent vector to the
oriented geodesic $\ora{\hat{\a}_i}$ at its starting point
clockwise to the tangent vector at the starting point of
$\ora{\hat{\a}_j}$.
\end{theorem}

The techniques are borrowed from Goldman \cite{goldman:hamiltonian}
and they could be adapted to treat surfaces with boundary or
surfaces with conical points and boundary. In fact, the formula
is manifestly the analytic continuation of its cousin in
\cite{mondello:WP}, obtained using techniques of Wolpert
\cite{wolpert:symplectic} and the doubling construction (unavalaible
here).

\subsection{Acknowledgements}
I would like to thank Martin M\"oller and Stefano Francaviglia
for fruitful discussions and an anonymous referee for useful remarks.

\section{Surfaces with constant nonpositive curvature}\label{sec:admissible}

\begin{definition}
A {\bf pointed surface} $(S,x)$ is a compact oriented surface $S$
of genus $g$ with a nonempty collection $x=(x_1,\dots,x_n)$
of $n$ distinct points on $S$. We will also write $\dot{S}$ for
the punctured surface $S\setminus x$.
\end{definition}
 
We will always assume that $n\geq 3$ if $g=0$.

Call $\La(S,x)$ the space of
{\bf $(S,x)$-admissible angle parameters},
made of $n$-tuples $\th=(\th_1,\dots,\th_n)\in\R_{\geq 0}^n$
such that
\[
\chi(\dot{S},\th):=(2-2g-n)+\sum_j\frac{\th_j}{2\pi}
\]
is nonpositive and we let $\La_-(S,x)$ (resp. $\La_0(S,x)$)
be the subset of admissible {\bf hyperbolic} (resp. {\bf flat})
angle parameters, namely those satisfying $\chi(\dot{S},\th)<0$
(resp. $\chi(\dot{S},\th)=0$).

We define $\La^\circ(S,x)=\La(S,x)\cap (\R\setminus 2\pi\N)^n$
and similarly $\La_0^\circ:=\La_0\cap \La^\circ$ and
$\La_-^\circ=\La_-\cap \La^\circ$.
Finally, $\La_{sm}(S,x):=\La(S,x)\cap[0,\pi)^n$ is the
subset of {\bf small} angle data.

\begin{definition}
An {\bf $\th$-admissible metric} $g$ on $(S,x)$
is a Riemannian metric of constant curvature on $\dot{S}$
such that, locally around $x_j$,
\[
g=\begin{cases}
f(z_j)|z_j|^{2r_j-2}|dz_j|^2 & \text{if $r_j>0$ or $\chi(\dot{S},\th)=0$} \\
f(z_j)|z_j|^{-2}\log^2|1/z_j|^2|dz_j|^2 & \text{if $r_j=0$ and
$\chi(\dot{S},\th)<0$} 
\end{cases}
\]
where $r_j=\th_j/2\pi$, $z_j$ is a local conformal coordinate at $x_j$
and $f$ is a smooth positive function.
A metric $g$ is {\bf admissible} if it is $\th$-admissible for
some $\th$.
\end{definition}

\begin{remark}
Notice that, if $\chi(\dot{S},\th)<0$ (or $\th\in\R_+^n$), then
such admissible metrics have finite area.
\end{remark}

Existence and uniqueness of metrics of nonpositive constant curvature
was proven by McOwen \cite{mcowen:point-singularity}
\cite{mcowen:prescribed} and Troyanov \cite{troyanov:euclidiennes}
\cite{troyanov:prescribing}.

\begin{theorem}[McOwen, Troyanov]\label{thm:metrics}
Given $(S,x)$ and an admissible $\th$ as above, there exists
a metric of constant curvature on $S$ and assigned
angles $\th$ at $x$ in each conformal class.
Such metric is unique up to rescaling.
\end{theorem}

Moreover, Schumacher-Trapani \cite{schumacher-trapani:cone} showed
that, for a fixed conformal structure on $S$, the restriction
to a compact subset $K\subset \dot{S}$ of the hyperbolic metric depends
smoothly on the associated admissible angle data,
provided $\th\in(0,2\pi)^n$.

\section{Spaces of admissible metrics}\label{sec:spaces}

Given a pointed surface $(S,x)$, consider the space of all
Riemannian metrics on $\dot{S}$, which is naturally an open convex
subset of a Fr\'echet space.
Let $\AMet(S,x)\subset\Met(S,x)$ be its
subspaces of admissible metrics and of metrics with conical singularities
at $x$. We will deliberately be sloppy about the regularity of such metrics.

The group $\Diff_+(S,x)$ of orientation-preserving diffeomorphisms
of $S$ that fix $x$ pointwise clearly acts on $\Met(S,x)$ preserving
$\AMet(S,x)$.

\begin{definition}
The {\bf Yamabe space} $\Yhat(S,x)$ is the quotient
$\AMet(S,x)/\Diff_0(S,x)$, where $\Diff_0(S,x)\subset
\Diff_+(S,x)$ is the subgroup of isotopies relative to $x$.
Moreover, $\Yy(S,x):=\Yhat(S,x)/\R_+$, where
$\R_+$ acts by rescaling.
\end{definition}

\begin{remark}
The definition above is clearly modelled on that of Teichm\"uller space
$\Teich(S,x)$, which is obtained as a quotient of the space of
conformal structures $\Conf(S,x)$ on $S$ by $\Diff_0(S,x)$.
\end{remark}

The {\bf mapping class group} $\Mod(S,x):=\Diff_+(S,x)/\Diff_0(S,x)$
acts on $\Yhat(S,x)$, on $\Yy(S,x)$ and on $\Teich(S,x)$.

There are two natural forgetful maps. The former
$\mathfrak{F}:\AMet(S,x)\lra\Conf(S,x)$
only remembers the conformal structure and the latter
$\Theta':\AMet(S,x)\lra \La(S,x)$ remembers the angles
at the conical points $x$.
They induce $F:\Yy(S,x)\lra\Teich(S,x)$ and
$\Theta:\Yy(S,x)\lra\La(S,x)$ respectively.
If $A\subset\La(S,x)$, then we will denote
$\Theta^{-1}(A)\subset\Yy(S,x)$ by $\Yy(S,x)(A)$
for brevity.

\begin{remark}
The forgetful map $(\widetilde{\mathfrak{F}},
\widetilde{\Theta}'):\Met(S,x)\lra\Conf(S,x)
\times\La(S,x)$
can be given the structure of a fibration
in Fr\'echet or Banach spaces
(see for instance \cite{schumacher-trapani:variation}).
\end{remark}

Theorem~\ref{thm:metrics} says that the
restriction of $(\widetilde{\mathfrak{F}},
\widetilde{\Theta}')$ to $\AMet(S,x)$ is a homeomorphism
and so its inverse is a section. The following
result (due to Schumacher-Trapani)
investigates the regularity of this section
and uses techniques of implicit function theorem.
%
%

\begin{theorem}[\cite{schumacher-trapani:variation}]
The homeomorphism
$(\mathfrak{F},\Theta'):\AMet(S,x)\lra \Conf(S,x)\times \La(S,x)$
restricts to a principal $\R_+$-fibration
over $\Conf(S,x)\times(\La_-(S,x)\cap(0,2\pi)^n)$,
and so does $\Yhat(S,x)\lra \Teich(S,x)\times\La(S,x)$.
Hence, $(F,\Theta):\Yy(S,x)\lra\Teich(S,x)\times\La(S,x)$
restricts to a $\Mod(S,x)$-equivariant homeomorphism
over $\Teich(S,x)\times(\La_-(S,x)\cap(0,2\pi)^n)$.
\end{theorem}

A deeper inspection of their proof
might show that $(\mathfrak{F},\Theta')$
restricts to an $\R_+$-fibration over
$\Teich(S,x)\times \La^\circ(S,x)$.
In this case, if $\Yy(S,x)(\La^\circ(S,x))$ is given the smooth structure
coming from Theorem~\ref{thm:holonomy}(a), then $(F,\Theta)$
would restrict to a $\Mod(S,x)$-equivariant
diffeomorphism over $\Teich(S,x)\times\La^\circ(S,x)$.

\section{Projective structures and holonomy}\label{sec:projective}

Let $h_{\k}=\k |dw|^2+|dz|^2$ be a Hermitian product on $\C^2$,
with $\k\leq 0$, and call $\PU_{\k}\subset \PSL_2(\C)$ the projective
unitary group associated to $h_{\k}$.

Given a pointed surface $(S,x)$, we denote by $\wti{\dot{S}}\rar\dot{S}$
its universal cover and by $\P T\dot{S}\rar\dot{S}$ and
$\P T\wti{\dot{S}}\rar\wti{\dot{S}}$
the bundles of real oriented tangent directions.
If $\dot{S}$ is endowed with a Riemannian metric, then $\P T\dot{S}$
identifies to the unit tangent bundle $T^1\dot{S}$.

Given an admissible metric $g$ on $(S,x)$ with angles $\th$
and curvature $\k$, one
can construct a {\bf developing map} so that the following diagram
\[
\xymatrix{
\P T\wti{\dot{S}} \ar[r] \ar[d] & \PU_{\k} \ar@{^(->}[r] \ar[d] &
\PGL_2(\C) \ar[d] \\
\wti{\dot{S}} \ar[r]^{\dev\quad} & D\qu \PU_{\k}
\ar@{^(->}[r] \ar[d]^{\cong} & B\qu \PGL_2(\C) \ar[d]^{\cong} \\
& \{v=w/z\in\C\,|\,|v|<1/\sqrt{|\k|}\} \ar@{^(->}[r] & \C\Pr^1
}
\]
commutes, where
$B\subset\PGL_2(\C)$ is the subset of upper triangular matrices
and $D=B\cap\PU_{\k}$. In fact,
the sphere $S_\k:=\{(w,z)\in\C^2\,|\,\k|w|^2+|z|^2=1\}$ is acted on
by $\mathrm{U}_\k$ transitively
and its projectivization $\Omega_k:=\Pr S_\k$ is still acted on
by $\PU_\k$. Hence, $\Omega_\k=D\qu\PU_{\k}$
comes endowed with a metric of curvature $\k$,
so that $\dev$ becomes a local isometry.

\begin{remark}
The group $\PU_\k$ preserves $h_\k$ and clearly all its nonzero
(real) multiples.
For $\k<0$, the couple $(\Omega_\k,\PU_\k)$ is isomorphic to
$(\Omega_{-1},\PU_{-1})$ and so to
$(\H,\PSL_2(\R))$. But $D\qu\PU_0=\{|z|=1\}$ and
$\k^{-1}h_\k\rar|dw|^2$ as $\k\rar 0$.
Hence,
$\Omega_0\cong
\{[w:z]\in\mathbb{CP}^1\,|\,z\neq 0\}\cong\C$ with
the Euclidean metric and
\[
\PU_0\cong\left\{\left(
\begin{array}{cc}
u & 0 \\
t & 1
\end{array} \right)\,\Big|\, u\in \mathrm{U}(1),\ t\in\C\right\}=
\{
v\mapsto uv+t\,|\,u\in\mathrm{U}(1),\ t\in\C
\}
\]
We conclude that $(\Omega_0,\PU_0)$ is isomorphic to
$(\R^2,\SE_2(\R))$, where $\SE_2(\R)$ is the group of affine isometries
of $\R^2$ that preserve the orientation.
\end{remark}

Let $\Pp(S,x)$ be the space of
{\bf moderately singular projective structures}
on $\dot{S}$ (up to isotopy), that is of those
whose Schwarzian derivative
with respect to the Poincar\'e structure corresponding to $\th=0$
has at worst double poles at $x$.
The fibration $p:\Pp(S,x)\lra\Teich(S,x)$
that only remembers the complex structure on $S$ is naturally
a principal bundle under the vector bundle
$\Qq(S,2x)\lra \Teich(S,x)$ of holomorphic quadratic differentials
(with respect to a conformal structure on $S$) with at worst double
poles at $x$.

We also call $\Pp_{con}(S,x)$ the space of
{\bf projective structures with conical points}, which
are defined to be those moderately singular projective structures
that satisfy the following condition: for every $j$ there exists
a local holomorphic coordinate around $x_j$ such that,
around $x_j=\{z_j=0\}$,
the universal covering map
$\wti{\dot{S}}\cong\H_{w_j}\rar\dot{S}$ can be written as
$w_j\mapsto \mathrm{exp}(iw_j)=z_j$ and the
developing map is conjugated to
$w_j\mapsto \mathrm{exp}(i r_j w_j)$ for $r_j>0$
(or to $w_j\mapsto w_j$, if $r_j=0$).
Projective structures with conical points,
that admit a developing map whose image is contained in $\Omega_\k$
and whose monodromy is a subgroup of $\PU_\k$,
are called {\bf admissible} and form a subspace $\Pp_{adm}(S,x)$.

\begin{lemma}
Projective structures with conical points are moderately singular
and the Schwarzian derivative between projective structures
with the same angle data have zero quadratic residue.\\
Hence, every hyperbolic metric with conical points
induces an admissible projective structure.
Moreover,
\[
\xymatrix{
\Yhat(S,x) \ar[rr]^{\Dhat} \ar[rd] && \Pp(S,x) \\
& \Yy(S,x)\ar[ru]_{\Dd}
}
\]
commutes, $\Dd$
is a homeomorphism onto $\Pp_{adm}(S,x)$, which is a closed real-analytic
subvariety. Finally, the restriction of $\Dd$ to
each slice $\Dd_\th:\Yy(S,x)(\th)\lra \Pp(S,x)$
is a homeomorphism onto a real-analytic subvariety of $\Pp_{adm}(S,x)$.
\end{lemma}

\begin{proof}
Admissibility is a simple computation: it turns out that
the Schwarzian derivative (with respect to the Poincar\'e
structure with cusps at $x$)
can be written as
\[
\Sch = \left[-\frac{1}{2}\left( \frac{\th_j}{2\pi}\right)^2
+O\left (z_j\right)\right]\,
\frac{dz_j^2}{z_j^2}
\]
where $z_j$ is a local holomorphic coordinate around $x_j$.
Notice also that the Schwarzian derivative of a
projective structure with conical singularities $\th$
with respect to another projective structure with conical
singularities $\wti{\th}$
looks like
\[
\Sch=\left[\frac{1}{2}\left( \frac{\wti{\th}_j^2-\th_j^2}{(2\pi\th_j)^2}
\right)
+O\left (z_j\right)\right]\,
\frac{dz_j^2}{z_j^2}
\]
around $x_j$ (the expression is valid also for $\wti{\th}_j=0$
and $\th_j>0$). This proves the claim on the residue of $\Sch$.

As the metric can be obtained up to scale by pulling back
the metric of $\Omega_\k$ via $\mathrm{Dev}$, it follows
that $\Dd$ is bijective. It is easy to check that $\Dd$
and $\Dd^{-1}$ are continuous.

Finally, observe that admissible projective structures
are characterized by the fact that the image of $\dev$ sits in $\Omega_\k$
and it has conical singularities at $x$.
The former is a real-analytic closed condition, that can be locally
rephrased in terms of holonomy in $\PU_\k$.
The latter is also a closed real-analytic condition
that can be phrased in terms of quadratic
residues of Schwarzian derivative (with respect
to the Poincar\'e structure with cusps at $x$).
A similar argument holds for the image of $\Dd_\th$.
\end{proof}

\begin{remark}
It can be proven that $\Pp_{adm}(S,x)$ is smooth and that
the natural map $\AMet(S,x)\lra\Pp_{adm}(S,x)$ is smooth
and submersive, which authorizes
to put on $\Yy(S,x)$ the smooth structure induced by $\Pp_{adm}(S,x)$.
Thus, $\Yhat(S,x)$ has a smooth structure too.
\end{remark}

Clearly, chosen a base point in $\dot{S}$,
we also have an associated {\bf holonomy representation}
\[
\rho: \gru:=\pi_1(\dot{S}) \lra \PU_\k
\]
whose image is discrete, for instance, if
each $\th_j=2\pi r_j$
with $1/r_j\in \N_+$. However, for almost all angles $\th$
the representation $\rho$ does not have discrete image.
%
%

Given a Lie group $G$,
call $\Rep(\gru,G)$ the space $\mathrm{Hom}(\gru,G)/G$
of representations up to conjugation.

We will denote by $\Hol$ the holonomy map
$\Hol:\Pp(S,x)\lra \Rep(\gru,\PGL_2(\C))$
(and by abuse of notation, its compositions
$\Yhat(S,x)\rar\Yy(S,x)\lra\Rep(\gru,\PGL_2(\C))$
with $\Dd$)
and by $\hol$ its ``restricted'' versions
$\hol:\Yy(S,x)(\La_-)\lra \Rep(\gru,\PSL_2(\R))$
and $\hol:\Yy(S,x)(\La_0)\lra\Rep(\gru,\mathrm{SE}_2(\R))$,
obtained using the isomorphisms $\PU_\k\cong\PSL_2(\R)$
and $\PU_0\cong\mathrm{SE}_2(\R)$.

Notice that the traces of the holonomies of the boundary
loops do not detect
the angles $\th\in\R^n$ at the conical points (with the exception
of the cusps), but just their class in $(\R/2\pi\Z)^n$. Thus,
we have a commutative diagram
\[
\xymatrix{
\Yy(S,x) \ar[r]^{\Dd} &
\Pp_{con}(S,x) \ar[d]^{\Hol}\ar[r]^{\Theta} & \R_{\geq 0}^n \ar[d]\\
& \Rep(\gru,\PGL_2(\C)) \ar[r]^{\ol{\Theta}} & (\R/2\pi\Z)^n
}
\]

\begin{theorem}\label{thm:holonomy}
The holonomy maps
satisfy the following properties:
\begin{itemize}
\item[(a)]
the restriction $\hol:\Pp_{adm}(S,x)\rar \Rep(\gru,\PGL_2(\C))$ to
$\Theta^{-1}(\La^\circ)$ is a real-analytic immersion
and so $\Theta^{-1}(\La^\circ)$ is smooth;
\item[(b)]
$\hol\Big|_{\La_{sm,-}}$ and
$\hol\Big|_{\La_{sm,0}}$ are injective onto open subsets
of the corresponding representation spaces.
\end{itemize}
Hence,
$\hol\Big|_{\La_{sm,-}}$ and $\hol\Big|_{\La_{sm,0}}$
are diffeomorphisms onto their images.
\end{theorem}

As a consequence,
$\Hol\Big|_{\La_{sm,-}}$ and $\Hol\Big|_{\La_{sm,0}}$
are diffeomorphisms onto their images too.

\begin{proof}
Part (a) was established by Luo \cite{luo:holonomy} in greater
generality.
In the flat case, it was already known to Veech \cite{veech:flat}.
Proposition~\ref{prop:coord-holonomy} gives a proof for the
hyperbolic and flat case that uses lengths of arcs dual to the spine.

Part (b) is a consequence of Lemma~\ref{lemma:geodesic},
which guarantees that there exists a (unique) smooth geodesic in
each homotopy class of simple closed curves, if the angles are smaller
than $\pi$, and that its length can be computed from the holonomy
representation. Thus, the injectivity follows from the standard reconstruction
principle for hyperbolic surfaces which are decomposed into a union of
pair of pants.
\end{proof}

Actually, a more careful look shows that, in negative curvature,
if $\th_j\leq \pi$ for every $j$,
then pair of pants decompositions still exist,
the reconstruction principle works
and the holonomy map is still injective.
Of course, one must allow ``degenerate pair of pants'' consisting
of one segment, which are obtained by cutting along a simple closed
geodesic which separates a couple $\{x_i,x_j\}$ with $\th_i=\th_j=\pi$ from
the rest of the surface and which consists of twice a geodesic
segment that joins $x_i$ and $x_j$.

Even though we will not formalize this approach here,
it is intuitive that the failure of the injectivity for
$\hol\Big|_{\th}$ is related to the lack of properness
of $\hol\Big|_{\th}$ and so
to the possibility of extending the holonomy map to some points
in the boundary of the augmented
Teichm\"uller space $\ol{\Teich}(S,x)$ in such a way that the holonomy
of a pinched loop is sent to an elliptic element of $\PSL_2(\R)$.

In fact,
if $J\in\ol{\Teich}(S,x)$, then $\hol\Big|_\th$ for hyperbolic metrics
continuously extends to $J$ if and only if
we can associated to $J$ a $\th$-admissible metric
$g$ in which the only type of degeneration is given by conical points
$x_{i_1},\dots,x_{i_k}$ with $\th_{i_1}+\dots+\th_{i_k}>2\pi (k-1)$
coalescing together.
When this singularity occurs,
the loop surrounding the coalescing points
has elliptic holonomy.

Hence, if there are $i_1,\dots,i_k$ such that $\th_{i_1}+
\dots+\th_{i_k}>2\pi(k-1)$, then
{\it the holonomy map $\hol\Big|_{\th}$
is not proper}, but it will become so if we extend it
to those points of $\ol{\Teich}(S,x)$ corresponding to the degenerations
mentioned before.

In the flat case, the situation is different as we don't have a collar
lemma (see Lemma~\ref{lemma:collar}),
so that injectivity may fail for arbitary small angles.
However, as in the hyperbolic case,
we do not have properness of the holonomy map
if $\th_i+\th_j>2\pi$ for certain $i\neq j$
(or if $\th_1>2\pi$ and $n=1$).

As an example of the non-injectivity phenomenon we have the following.

\begin{proposition}\label{prop:injectivity}
{\rm{(a)}}
Let $\th\in\La_-$ be angle data such that $\th_h+\th_j>2\pi$
for certain $h\neq j$.
Then $\hol\Big|_\th$ is not injective.\\
{\rm{(b)}}
Let $\th\in\La_0$ be angle data such that $\th_h+\th_j\in (2\pi,\infty)
\cap\Q$ for certain $h\neq j$. Then $\hol\Big|_\th$ is not injective.
\end{proposition}

\begin{proof}
The case in which some angles are positive multiples of $2\pi$ are
treated in Proposition~\ref{prop:integral}, so that we can assume
that no holonomy along the loop $\g_k$ that winds around $x_k$
is the identity for all $k=1,\dots,n$.

Let's analyze case (a).
Because $\th_h+\th_j>2\pi$, there are metrics in which $x_h$ and $x_j$
are at distance $d>0$ arbitrarily small.
Given a metric $g$, we can assume up to conjugation that
\begin{align*}
\hol(g)(\g_h) & =
\left(
\begin{array}{cc}
\cos(\ti{\th}_h/2) & -\sin(\ti{\th}_h/2) \\
\sin(\ti{\th}_h/2) & \cos(\ti{\th}_h/2)
\end{array} 
\right) \\
\hol(g)(\g_j) & =
\left(
\begin{array}{cc}
\cos(\ti{\th}_j/2) & -e^d \sin(\ti{\th}_j/2) \\
e^{-d}\sin(\ti{\th}_j/2) & \cos(\ti{\th}_j/2)
\end{array}
\right)
\end{align*}
where $\ti{\th}_j,\ti{\th}_h\in(0,2\pi)$,
$\th_j\equiv\ti{\th}_j$ and $\th_h\equiv\ti{\th}_h$
mod $2\pi$.

Thus, the loop $\b:=\g_j\ast\g_h$ has
holonomy $\hol(g)(\b)=
\hol(g)(\g_h)\hol(g)(\g_j)$
with
\[
|\mathrm{Tr}(\hol(g)(\b))|=
2|\cos(\ti{\th}_h/2)\cos(\ti{\th}_j/2)-
\cosh(d)\sin(\ti{\th}_h/2)\sin(\ti{\th}_j/2)|
\]
which is strictly smaller than $2|\cos[(\ti{\th}_h+\ti{\th}_j)/2|\leq 2$.

Hence, there exists another metric $g'$ for which
such a $d>0$ is small and
$|\mathrm{Tr}(\hol(g')(\b))|=
2|\cos(\pi p/q)|$, where $p,q$ are positive coprime integers and $p/q<1$,
and so $\hol(g')(\b)$ has order $q$.

Let $\tau_\b\in\Mod(S,x)$ be the Dehn twist along $\b$.
If we place the basepoint for $\pi$ outside
the component of $S\setminus\b$ the contains $x_h$ and $x_j$,
then the action of
$\tau_\b$ on $\Rep(\gru,\PSL_2(\R))$ is trivial on every
loop that does not meet $\b$ and it is by conjugation by $\hol(\cdot)(\b)$
on $\g_h$ and $\g_j$.
Hence, $\tau_\b^q$ fixes $\hol(g')$ but it acts freely on $\Teich(S,x)$,
which shows that the holonomy map is not injective.

The proof of (b) follows the same lines, but it's actually easier.
In fact, $\hol(g)(\b)$ is actually a rotation of angle
exactly $\th_1+\th_2-2\pi$ (centered somewhere in the plane). Thus,
it is of order $q$. Hence, $\tau_\b^q$ acts trivially
on $\Rep(\gru,\SE_2(\R))$ but freely on $\Teich(S,x)$ and
the conclusion follows.
\end{proof}

A suitable modification of part (a) of the above proof would
also show that injectivity would similarly fail if $n=1$ and $\th_1>2\pi$.

Our feeling is that the non-injectivity of the holonomy map
in negative curvature is only associated to the phenomenon above.
It would be interesting to make this precise.

Another interesting issue is to understand when the images of $\hol\Big|_{\th}$
and $\hol\Big|_{\wti{\th}}$ intersect. For instance,
if all angles are integral multiples of $2\pi$, then the holonomy
representation descends to $\Rep(\pi_1(S),\PSL_2(\R))$
and Milnor-Wood's inequality allows us to recover
$\th_1+\dots+\th_n$. Given a $\th$-admissible hyperbolic
metric $g$, the question then becomes whether $\hol(g)$
remembers at least the area of $g$.

The last piece of information about the holonomy maps
concerns what happens when some angles are {\it integral},
i.e. integral multiples of $2\pi$, and so the corresponding holonomies
are the identity.

\begin{proposition}\label{prop:integral}
Let $G=\PSL_2(\R)$ (if $\chi(\dot{S},\th)<0$)
or $G=\SE_2(\R)$ (if $\chi(\dot{S},\th)=0$).
\begin{itemize}
\item[(1)]
If $\th\in \La^\circ(S,x)$,
then $\hol\Big|_{\th}:\Yy(S,x)(\th)\lra\Rep(\gru,G)$
is a locally closed real-analytic diffeomorphism onto its image.
\item[(2)]
If $\th_j= 2\pi$, then
$\hol\Big|_{\th}:\Yy(S,x)(\th)\cong\Teich(S,x)\lra
\Rep(\gru,G)$ is constant
along the fibers of the forgetful map
$\Teich(S,x)\rar\Teich(S,x\setminus\{x_j\})$.
\item[(3)]
If $\th_j=2\pi r_j$ with $r_j\geq 1$ integer
and if $z_j$ is a holomorphic coordinate
on $S$ around $x_j$ such that
locally $\dev(z_j)=z_j^{r_j}+b$, then
the differential of
$\hol\Big|_{\th}:\Pp_{con}(S,x)(\th)\lra\Rep(\gru,G)$
vanishes along the tangent directions
determined by deforming the local developing
map around $x_j$ as
$\dev_\e(z_j)=b+(z_j+\e cz_j^{1-r_j})^{r_j}+o(\e)=
b+z_j^{r_j}+r_j c\e+o(\e)$, for every $c\in \C$.
Hence, the differential of $\hol\Big|_{\th}:\Yy(S,x)(\th)\cong
\Teich(S,x)\lra \Rep(\gru,G)$ vanishes
along the first-order Schiffer variation $c z_j^{1-r_j}\frac{\pa}{\pa z_j}$.
\end{itemize}
\end{proposition}

We recall that a Schiffer variation of complex structure on $(S,J)$ is
defined as follows. Let $D_j\subset S$ be
a disc centered at $x_j$ and let $z_j$ be a
holomorphic coordinate on $D_j$ so that $z_j(D_j)=\{z\in\C\,|\,|z|<1\}$;
call $D_{i,\d}:=\{p\in D_j\,|\,|z_j(p)|<\d\}$.
Given a holomorphic vector field $V=f(z_j)\pa/\pa z_j$
on $\dot{D}_j$ with a pole in $x_j$,
we can define a new Riemann surface $(S_\e,J_\e)$
(which is canonically diffeomorphic
to $S$ up to isotopy) by gluing $D_j$ and
$(S\setminus D_{j,1/2})\cup g_\e(D_j)$ through the map
$g_\e:D_j\setminus D_{j,\d} \rightarrow S\setminus\{x_j\}$
given by
$z\mapsto z+\e f(z)$, which is a biholomorphism onto its image for
$\e$ small enough.

A simple argument shows that the tangent direction in
$T_J \Teich(S,x)\cong H^{0,1}_J(S,T_S(-x))$
determined by such a Schiffer variation does not depend on the disc
$D_j$ and on $\d$, but only on the jet of $V$ at $x_j$.
In particular, we have
\[
\xymatrix{
0 \ar[r] & H^0(S,T_S(-x+\infty x_j))\ar[r] &
\hat{\mathcal{M}}_{S,x_j}/\hat{\mathcal{O}}_{S,x_j}(T_S(-x))
\ar[r] & H^{0,1}(S,T_S(-x)) \ar[r] & 0
}
\]
where $\hat{\mathcal{O}}_{S,x_j}$ is the completed local ring
of functions at $x_j$ and $\hat{\mathcal{M}}_{S,x_j}$ is its field
of fractions. More naively, elements in
$\hat{\mathcal{M}}_{S,x_j}/\hat{\mathcal{O}}_{S,x_j}(T_S(-x))$
can be represented as $(\sum_{-m\leq k\leq 0} c_k z_j^k)\pa/\pa z_j$.

\begin{proof}[Proof of Proposition~\ref{prop:integral}]
Part (1) is clearly a consequence of Theorem~\ref{thm:holonomy}(a).

For part (3), notice that the holonomy around $x_j$ is trivial.
Thus, the vector field $c z_j^{1-r_j}\frac{\pa}{\pa z_j}$ that
deforms the local developing map as $z_j\mapsto (z_j+\e cz_j^{1-r_j})^{r_j}=
z_j^{r_j}+r_j c\e+o(\e)$ produces a deformation
of projective structure which fixes the holonomy.
Clearly, (2) follows from (3).
\end{proof}

\begin{remark}
Notice that a simultaneous Schiffer variation at $x_1,\dots,x_n$
with vector fields $V_1,\dots,V_n$ determine the zero tangent
vector only if they extend to a global section of $T_S$ (holomorphic
on $S\setminus x$), and this can happen only if
$m_1+\dots+m_n\geq 2g-2+n$, where $m_j=\mathrm{ord}_{x_j}(V_j)$.
Thus, if $\chi(\dot{S},\th)<0$ or if
$\th_j\notin 2\pi\N_+$ for some $j$,
then any first-order deformation of an
admissible metric that fixes holonomy
changes the conformal structure.
\end{remark}

\section{Poisson structures}\label{sec:poisson}

Now, we will implicitly represent each class in $\Yy(S,x)(\La_-)$
by a metric $g$ of curvature $-1$, so that the (restricted) holonomy map
gives a representation $\rho:\gru=\pi_1(\dot{S})\lra\PSL_2(\R)$.
Because of the choice of a base-point, $\rho$ is only well-defined
up to conjugation by $\PSL_2(\R)$.

On the other hand, we also have a local system $\xi\lra \dot{S}$ defined
by $\xi=(\wti{\dot{S}}\times\lieg)/\gru$,
where $\wti{\dot{S}}$ is the universal cover of $\dot{S}$,
$\lieg=\psl_2(\R)$ is the Lie algebra of $\PSL_2(\R)$ and
$\gru$ acts on $\wti{\dot{S}}$ via deck transformations
and on $\lieg$ via $\rho$ and the adjoint representation.
Let $D_1,\dots,D_n\subset S$ be open disjoint discs such that $x_j\in D_j$
and call $D=\bigcup_j D_j$. We will slightly abuse notation by
denoting still by $\xi$ the restriction of $\xi\rar\dot{S}$ to $\dot{D}$.

We recall that $\Bb(X,Y):=\tr(XY)$ for
$X,Y\in\lieg$ is a nondegenerate symmetric bilinear form
of signature $(2,1)$. Given
\[
H=\left(
\begin{array}{cc}
1 & 0\\
0 & -1
\end{array}
\right)
\qquad
E
=\left(
\begin{array}{cc}
0 & 1\\
0 & 0
\end{array}
\right)
\qquad
F
=\left(
\begin{array}{cc}
0 & 0\\
1 & 0
\end{array}
\right)
\]
then $\{H,E+F,E-F\}$ is a $\Bb$-orthogonal basis of $\lieg$,
with $\Bb(H,H)=\Bb(E+F,E+F)=2$ and $\Bb(E-F,E-F)=-2$.
Notice that $E-F$ generates the rotations around $i\in\H$.
Actually, $\mathfrak{K}=-4\Bb$,
where $\mathfrak{K}$ is the Killing form on $\lieg$.
Denote still by $\Bb$ the induced
pairing on $\lieg^*$.

Deforming the (conjugacy class of the)
representation $\rho$ is equivalent to
deforming the (isomorphism class of the) local system $\xi$.

As shown for instance in \cite{goldman:symplectic}, first-order
deformations of $\rho\in\Rep(\gru,\PSL_2(\R))$
are parametrized by $H^1(\dot{S};\xi)$.
Thus, $T_{\rho}\Rep(\gru,\PSL_2(\R))\cong H^1(\dot{S};\xi)$ and dually
$T^*_{\rho}\Rep(\gru,\PSL_2(\R))\cong H_1(\dot{S};\xi^*)$,
which is isomorphic to $H^1(\dot{S},\dot{D};\xi)$ by
Lefschetz duality (and the nondegeneracy of $\Bb$).

When no $\th_j\in 2\pi\N_+$,
the long exact sequence in cohomology for
the couple $(\dot{S},\dot{D})$
give rise to the following identifications
\[
\xymatrix@R=0.2cm{
0 \ar[r] &
H^0(\dot{D};\xi) \ar[r] \ar@{=}[d]^\wr &
H^1(\dot{S},\dot{D};\xi) \ar[r] \ar@{=}[d]^\wr &
H^1(\dot{S};\xi) \ar[r] \ar@{=}[d]^\wr &
H^1(\dot{D};\xi) \ar[r] \ar@{=}[d]^{\wr} &
0
\\
0 \ar[r] &
(\R^n)^* \ar[r]^{(d\ol{\Theta})^*\qquad} &
T^*_{\rho}\Rep(\gru,\PSL_2(\R)) \ar[r]^{\eta} &
T_{\rho}\Rep(\gru,\PSL_2(\R)) \ar[r]^{\qquad d\ol{\Theta}} &
\R^n \ar[r] &
0
}
\]
where $g\in\Yy(S,x)(\La_-^\circ)$ and
$H^0(\dot{S};\xi)\cong H^2(\dot{S},\dot{D};\xi)^*=0$
because $\rho$ has no fixed vectors.

Notice that, if $g$ is a $\th$-admissible
metric, the parabolic cohomology group
$H^1_P(\dot{S};\xi)$ at $\rho=\hol(g)$, defined as the image of
$H^1(\dot{S},\dot{D};\xi) \rar H^1(\dot{S};\xi)$
identifies (via $\hol$) to the space of those first-order
deformations of metrics (equivalently,
of projective structures) with conical
singularities along which $\th$ is constant.

At a point $\rho$ such that $\th_j\in 2\pi\N_+$,
we have $H^0(\dot{D}_j;\xi)\cong
\lieg\cong H^1(\dot{D}_j;\xi)$ and so $\Rep(\gru,\PSL_2(\R))$
is singular at such a $\rho$. 
In this case, there are deformations of $\rho$
which correspond to opening a hole or creating a cusp at $x_j$.
Conversely, if no $\th_j\in 2\pi\N_+$, then
$\hol(g)$ lies in the smooth locus of $\Rep(\gru,\PSL_2(\R))$.

Though not completely trivial,
the following result can be obtained adapting arguments from
\cite{atiyah-bott:yang-mills},
\cite{goldman:symplectic} or \cite{karshon:algebraic}, who proved
that $\eta$ defines a symplectic structure if $x$ is empty.

\begin{lemma}
The alternate pairing $\eta$ defines a Poisson structure
on the smooth locus of $\Rep(\gru,\PSL_2(\R))$.
Hence, the pull-back of $\eta$ through $\hol$
defines a Poisson structure on $\Yy(S,x)(\La_-^\circ)\cong
\Teich(S,x)\times\La_-^\circ(S,x)$,
which will still be denoted by $\eta$.
\end{lemma}

The second part follows from the fact that $\hol$ is a local
diffeomorphism (Theorem~\ref{thm:holonomy}(a)).

As already investigated by Goldman \cite{goldman:symplectic} in the
case of closed surfaces, it is natural to explore the relation
between $\eta$ and the {\bf Weil-Petersson pairing},
which is defined as $\eta_{WP,\th}:=\mathrm{Im}(h^*_{WP})$, where
\[
h^*_{WP,\th}(\varphi,\psi):=-\frac{1}{4}\int_S g_\th^{-1}(\varphi,\ol{\psi}) 
\]
$g_\th^{-1}$ is the dual hyperbolic K\"ahler form on $S$ with angle
data $\th$ and $\varphi,\psi\in H^0(S,K_S^{\otimes 2}(x))$
are cotangent vectors to $\Teich(S,x)\cong \Yy(S,x)(\th)$ at $g$.

For angles smaller than $2\pi$, the Shimura isomorphism still holds.

\begin{theorem}\label{thm:eta}
If $\th\in \La_-(S,x)\cap(0,2\pi)^n$, then
\[
\eta_{WP,\th}=-\frac{1}{8}\eta\Big|_{\th}
\]
as dual symplectic forms on $\Yy(S,x)(\th)\cong\Teich(S,x)$.
\end{theorem}

Schumacher-Trapani \cite{schumacher-trapani:cone} have also shown
that, if $\th\in(0,2\pi)^n$, then $\eta^*_{WP,\th}$ is a K\"ahler form
and that $\eta^*_{WP,\th}$ degenerates in the expected way as
some $\th_j\rar 2\pi$.

\begin{proof}[Proof of Theorem~\ref{thm:eta}]
Mimicking \cite{goldman:symplectic}, we consider the diagram
\[
\xymatrix{
& \xi=\dev^*\lieg \ar[d]^{\dev^*\sigma} \\
T_{\dot{S}} \ar[r]^{\b\quad} & \dev^* T_{\H}
}
\]
in which $\sigma:\lieg\rar T_{\H}$ maps
$\lieg$ to the $\mathrm{SL}_2(\R)$-invariant
vector fields of $\H$.

If $r_j=\th_j/2\pi>0$, then $\dev$ locally looks like
\[
\dev : \ z_j \mapsto i\,\frac{1-z_j^{r_j}}{1+z_j^{r_j}}
\]
up to action of $\PSL_2(\R)$
for some holomorphic local coordinate $z_j$ around $x_j$.
So
\[
\b:=d(\dev)=-\frac{2ir_j z_j^{r_j-1}}{(1+z_j^{r_j})^2}
\]

Moreover, if $w$ is the standard coordinate on
$\H=\{w=s+it\,|\,s,t\in\R,\ t>0\}$, then
\[
\Bb\sigma=\left(
\begin{array}{cc}
w & -w^2\\
1 & -w
\end{array}
\right)\frac{\pa}{\pa w}
\]
thus, around $z_j=0$ we have that
\begin{align*}
\tau:=\b^{-1}\circ\dev^*(\Bb\sigma)
& =
\Big[i(1-\frac{2z_j^{r_j}}{1+z_j^{r_j}})H
+(2-\frac{4z_j^{r_j}}{(1+z_j^{r_j})^2})(E+F)+\\
& \qquad\qquad -\frac{4z_j^{r_j}}{(1+z_j^{r_j})^2}(E-F)\Big]
\frac{i(1+z_j^{r_j})^2}{2r_j z_j^{r_j-1}}\frac{\pa}{\pa z_j}
\end{align*}
belongs to
$H^0(S,T_S(\sum_j(r_j-1)x_j)\otimes\xi)$.
Moreover, the dual K\"ahler form associated to the Poincar\'e
metric on $\H$
\[
g_{\H}^{-1}=t^2 \frac{\pa}{\pa s}\wedge \frac{\pa}{\pa t}=
-2it^2\frac{\pa}{\pa w}\wedge\frac{\pa}{\pa \ol{w}}
\]
can be recovered as $g_{\H}^{-1}=(i/2)\mathrm{Tr}(\Bb\sigma
\wedge\Bb\ol{\sigma})$, where
\[
\Bb\sigma\wedge\Bb\ol{\sigma}=
\left(
\begin{array}{cc}
|w|^2-w^2 & (w-\ol{w})|w|^2 \\
\ol{w}-w & |w|^2-\ol{w}^2
\end{array}
\right)
\frac{\pa}{\pa w}\wedge\frac{\pa}{\pa \ol{w}}
\]
Hence, $g_\th^{-1}=-(i/2) \Bb(\tau\wedge\ol{\tau})$.

As we can identify $T\Pp_{adm}(S,x)$ and $T\Rep(\gru,\PSL_2(\R))
\cong H^1(\dot{S};\xi)$
via $d\hol$, then the restriction of $p:\Pp(S,x)\rar\Teich(S,x)$
to $\Pp_{adm}(S,x)$ can be infinitesimally described as follows.
Given $\nu\in H^1(\dot{S};\xi)$, we can look at its restrictions
$\nu_j\in H^1(\dot{D}_j;\xi)$. If $\nu_j=0$, then $\nu$ does not
vary the angle $\th_j$ and so there is a representative for $\nu$
that vanishes on $\dot{D}_j$.
If $\nu_j\neq 0$, then
it can be represented by a Cech
$1$-cocycle with locally costant coefficients in $\xi$.
%
%
As the $(E-F)$-component of $\tau$ is $\displaystyle
-\frac{2iz_j}{r_j}\frac{\pa}{\pa z_j}$, we conclude
that $\tau\nu_j$ has a representative that vanishes at $x_j$.
Hence, $\tau\nu$ has always a representative that vanishes at $x$,
whose class in $H^1(S,T_S(-x))$ will be denoted by $\wti{\tau\nu}$,
and $dp:T\Rep(\gru,\PSL_2(\R))\rar\Teich(S,x)$ incarnates into
\[
\xymatrix@R=0.1cm{
H^1(\dot{S};\xi)\ar[r] & H^1(S,T_S(-x)) \\
\nu \ar@{|->}[r] & \Bb(\wti{\tau\nu})
}
\]
which is the restriction to real projective structures of the map
$H^1(\dot{S};\xi_{\C})\lra H^1(S,T_S(-x))$ still given 
by $\nu\mapsto \Bb(\wti{\tau\nu})$.
Its dual is thus
\[
\xymatrix@R=0.1cm{
H^0(S,K_S^{\otimes 2}(x)) \ar[r] & H^1(\dot{S},\dot{D};\xi_{\C})\\
\varphi \ar@{|->}[r] &  \wti{\varphi\tau}
}
\]
where $\wti{\varphi\tau}$ can be
represented by $\xi_{\C}$-valued $1$-form
cohomologous to $\varphi\tau$, which vanishes on $\dot{D}$,
whose existence depends on the fact that no $\th_j\in 2\pi\N$
and so $\varphi\tau$ has no residue at $x$.
A similar formula holds for real projective structures.

Hence, it is easy now to see that, if all the terms
are convergent, then
\begin{align*}
h^*_{WP,\th} & =-\frac{1}{4}\int_S g_{\th}^{-1}(\tau,\ol{\psi})=
\frac{i}{8}\int_S\Bb(\varphi\tau\wedge\ol{\psi\tau})=
\frac{i}{8}\int_S\Bb(\wti{\varphi\tau}\wedge\ol{\wti{\psi\tau}})=\\
& =\frac{i}{8}[S]\cap\Bb(p^*(\varphi)\cup \ol{p^*(\psi)})
\end{align*}
As we are working with real projective structures, $\ol{\psi\tau}=\psi\tau$ and
this concludes the argument.
\end{proof}
%

Notice that, as $\th_j>2\pi$ increases, the Weil-Petersson pairing on
$T\Yy(S,x)(\th)$ becomes more and more degenerate, the walls being given
exactly by $\th_j\in 2\pi\N$.
However, the above proof also yields the following.

\begin{corollary}
If $\th\in\La_-(S,x)$ and $\varphi,\psi\in T^*\Teich(S,x)$, then
\[
\eta_{WP,\th}(\varphi,\psi)=\frac{1}{8}\eta\Big|_{\th}(\varphi,\psi)
\]
whenever both hand-sides are convergent.
\end{corollary}

\section{Decorated hyperbolic surfaces}\label{sec:decorated}

Let $\th_{max}=
\mathrm{max}\{\th_1,\dots,\th_n\}$ and
recall the collar lemma for hyperbolic surfaces with conical points.

\begin{lemma}[Dryden-Parlier \cite{dryden-parlier:cone}]\label{lemma:collar}
If $\th\in\La_{sm,-}(S,x)$, then
there exists $R\in(0,1]$ which depends only on $\th_{max}<\pi$ such that,
for every hyperbolic metric $g$ on $S$ with angles $\th$ at $x$,
the balls $B_j$ centered at $x_j$ with circumference $\leq R$
are disjoint and do not meet any simple closed geodesic.
\end{lemma}

We call such balls $B_j$ {\it small}.
The following definition is inspired by Penner \cite{penner:decorated},
who first introduced decorated hyperbolic surfaces with cusps.
Notice that a class in $\Yy(S,x)$
will be usually represented by an admissible metric of curvature $-1$.

\begin{definition}
A {\bf decoration} for a hyperbolic surface $(S,x)$ with small angle data $\th$
is the choice of small balls $B_1,\dots,B_n$ (not all reduced to a point);
equivalently, of the nonzero vector $\e=(\e_1,\dots,\e_n)\in [0,R)^n$
of their circumferences.
\end{definition}

\begin{remark}\label{rmk:dec}
Notice that a hyperbolic surface $S$ with small angles $\th$
can be given a {\it standard decoration} by letting $B_j$ to be
the ball of radius $s(\th)=\cosh^{-1}(1/\sin(\th_{max}/2))/2$.
The constant is chosen in such a way that the area of
$B:=B_1\cup\dots\cup B_n$
is bounded from below (by a positive constant) for all hyperbolic
structures on $S$ (with angle $\th$).
The circumference of $B_j$ is clearly $s(\th)\th_j$.
\end{remark}

Thus, the assignment of $[s(\th)\th]$ defines
a map $\Yy(S,x)\setminus\Theta^{-1}(0)\lra \mathbb{P}(\R_{\geq 0}^n)$.
The closure of its graph identifies to the real-oriented
blow-up $\Bl_0\Yy(S,x)$
and the exceptional divisor $\Theta^{-1}(0)\times\mathbb{P}(\R_{\geq 0}^n)$
can be understood as the space of hyperbolic metrics with
cusps on $\dot{S}$ (up to isotopy) together with a
{\bf projective decoration} $[\e]\in\mathbb{P}(\R_{\geq 0}^n)$,
which plays the role of infinitesimal angle datum.
Clearly, a projective decoration $[\e]$ is canonically represented
by the {\bf normalized decoration} $\e$
in its class, obtained by prescribing
$\e_1+\dots+\e_n=1$; so we can identify $\mathbb{P}(\R_{\geq 0}^n)$
with $\Delta^{n-1}$.

Thus, the map $\Theta$ lifts to $\wh{\Theta}:\Bl_0\Yy(S,x)
\lra \Delta^{n-1}\times[0,2\pi(2g-2+n)]$.
We remark that a similar projective decoration arises in
\cite{mondello:WP} as infinitesimal boundary length datum.

%
%
%
%
%
\section{Arcs}\label{sec:arcs}

Given a pointed surface $(S,x)$, we call {\bf arc} the image
$\a=f(I)$ of a continuous $f:(I,\pa I)\rar(S,x)$, in which $I=[0,1]$ and $f$ 
injectively maps $\mathring{I}$ into $\dot{S}$.
Let $\Arc_0(S,x)$ be the space of arcs with the compact-open topology
and let $\Arc_n(S,x)$ be the subset of $\Arc_0(S,x)^{(n+1)}$ consisting
of unordered pairwise non-homotopic (relative to $x$) $(n+1)$-tuple of arcs
$\ua=\{\a_0,\dots,\a_{n}\}$ such that $\a_i\cap\a_j\subset x$ for $i\neq j$.

\begin{remark}
Equivalently, we could have defined $\Arc'_0(S,x)$ to be the space
of unoriented simple closed free loops $\g$ in $S\setminus x$ which are
homotopy equivalent to an arc $\a$ (i.e. such that $\g=\pa U_{\a}$,
where $U$ is a tubular neighbourhood of $\a$). We could have
defined $\Arc'_n(S,x)$ analogously. Clearly, $\Arc'_n(S,x)\simeq \Arc_n(S,x)$.
We will also say that $\a_1,\a_2\in\Arc_0(S,x)$ are homotopic
{\it as arcs} if they belong to the same connected component.
\end{remark}

Notice that each $\Arc_n(S,x)$ is contractible, because
$\chi(\dot{S})<0$.

\begin{definition}
A {\bf $(k+1)$-arc system} is an element of $\Af_k(S,x):=\pi_0(\Arc_k(S,x))$.
A {\bf triangulation} is a maximal system of arcs $\ua\in\Af_{N-1}(S,x)$,
where $N=6g-6+3n$.
\end{definition}

Notice that, if $\ua=\{\a_i\}$ is a triangulation,
then its {\bf complement} $S\setminus\ua:=S\setminus\bigcup_i \a_i$
is a disjoint union of triangles.

\begin{lemma}\label{lemma:geodesic}
Let $\a_i$ be an arc and $g$
be a $\th$-admissible metric on $(S,x)$.
\begin{itemize}
\item[(1)]
There exist a geodesic $\hat{\a}_i\subset S$ and a homotopy
$\a_i(t):I\rar S$ with
fixed endpoints such that $\a_i(0)=\a_i$, $\a_i(1)=\hat{\a}_i$ and
$\mathrm{int}(\a_i(t))\cap x=\emptyset$ for $t\in[0,1)$
and $\mathrm{int}(\a_i(1))\cap x$
can only contain points $x_j$ such that $\th_j\geq\pi$.
\item[(2)]
If two geodesic arcs $\hat{\a}_i$ and $\hat{\a}'_i$
are homotopic as arcs, then they are equal.
\item[(3)]
If all $\th_j <\pi$, then for each $\a_i$
there exists exactly one smooth geodesic $\hat{\a}_i$
homotopic to $\a_i$ as an arc.
\end{itemize}
\end{lemma}

The second assertion is a consequence of the nonpositivity of the curvature and
(3) follows from (1) and (2).
To prove (1), one takes a minimizing sequence in the homotopy
class of $\a_i$ and a limit $\hat{\a}_i$ of such a sequence
($S$ is compact). One immediately concludes by
looking at the geometry of a conical point. Whether or not
the (possibly broken)
geodesic $\hat{\a}_i$ obtained in (1) is an arc, we will still say
by abuse of notation that $\hat{\a}_i$ is
{\it the unique geodesic homotopic to $\a_i$}.

\begin{definition}
An arc $\a_i$ on $(S,x)$ is {\bf compatible} with the metric $g$
if there exists a smooth geodesic $\hat{\a}_i$, which is homotopic to $\a_i$
as arcs.
\end{definition}

Let $p\in\a^\circ_i\subset \dot{S}$ and let $\g_b,\g_c\in\pi_1(\dot{S},p)$
be loops that wind around $x_b,x_c$ such that $\g_b\ast\g_c$ corresponds
to $\a_i$.
If $\dev:\ti{\dot{S}}\rar\Omega$ is the developing map
(where $\Omega=\H,\C$), then
call $\ti{x}_b,\ti{x}_c$ the endpoints of $\ti{\a}_i:=\dev(\a'_i)$,
where $\a'_i$ is a lift of $\a_i$ to $\ti{\dot{S}}$.

\begin{definition}
The {\bf $a$-length} associated to an arc $\a_i$ is the
function $a_i:\Yhat(S,x)\lra[0,\infty]$ defined as the
distance between $\ti{x}_b$ and $\ti{x}_c$.
\end{definition}

\begin{remark}\label{rmk:a-lengths}
Notice that, if the angles at $x_b$ and $x_c$ are not integral
multiples of $2\pi$, then $\ti{x}_b$ and $\ti{x}_c$ are the unique fixed
points in $\H$ of $\Hol(g)(\g_b)$ and $\Hol(g)(\g_c)$.
Hence, Lemma~\ref{lemma:trig}(a) and Lemma~\ref{lemma:trig-flat}(a)
ensure that $a_i$ is real-analytic around $g$, where $a_i>0$
(i.e. where $\ti{x}_b\neq \ti{x}_c$).
Moreover, if $\a_i$ is compatible with $g$,
then $a_i(g)$ is the $g$-length of $\hat{\a}_i$.
In general, the length $\hat{a}_i$ of the (broken) geodesic
$\hat{\a}_i$ homotopic to $\a_i$ is positive and piecewise real-analytic:
in fact, $\hat{\a}_i$ is locally equal to the join of
finitely many smooth geodesic arcs $\hat{\a}_{i_1}\ast\dots\ast
\hat{\a}_{i_k}$ and so $\hat{a}_i=a_{i_1}+\dots+a_{i_k}$.
\end{remark}

Given a triangulation $\ua$,
the $a$-lengths associated to the unique
hyperbolic metric define a map
\[
\ell_{\ua}: \Yy(S,x)\lra \Bl_0 [0,\infty]^N
\]
where the infinitesimal $a$-lengths $\Delta^{N-1}$ arise in particular
when the surface becomes flat.

If $(S,x,B)$ is a surface with hyperbolic metric $g$,
small $\th$ and a normalized decoration $B$,
then we can define the {\bf reduced $a$-length} of an $\a_i$
that joins $x_b$ and $x_c$
to be $\tilde{a}_i:=a_i-(\e_b+\e_c)$, where $\e_b,\e_c$ are the radii
of $B_b,B_c$.
If $\a_i$ is compatible with $g$,
then $\tilde{a}_i=\ell_{\hat{\a}_i\setminus B}$.
Because of the standard decoration
mentioned in Remark~\ref{rmk:dec} for metrics with small angles,
the reduced $a$-lengths can be extended to an open neighbourhood of
$\wh{\Theta}^{-1}(0)$.

\begin{definition}
A triangulation $\ua$ of $(S,x)$ is {\bf adapted} to the $\th$-admissible
metric $g\in\Bl_0\Yy(S,x)$ if:
\begin{itemize}
\item[(a)]
every $\a_i\in\ua$ is compatible with $g$;
\item[(b)]
if $\th\neq 0$, then there is only one directed arc in $\ua$ outgoing
from each cusp (resp. from each cylinder, if $\chi(\dot{S},\th)=0$);
\item[(c)]
if $\th=0$ and $[\e]$ is the projective decoration, then
there is only one directed arc in $\ua$ outgoing from those $x_j$
with $\e_j=0$.
\end{itemize}
\end{definition}

We remark that, if $\th\in[0,\pi)^n$, then the compatibility
condition (a) is automatically satisfied. The utility of
adapted triangulations relies on the following result,
which directly follows from the above considerations.

\begin{proposition}\label{prop:coord-holonomy}
Let $\ua$ be triangulation adapted to $g\in \Yy(S,x)\setminus\Theta^{-1}(0)$
(resp. $(g,[\e])\in\Theta^{-1}(0)\subset\Bl_0\Yy(S,x)$) and suppose
that $\th_j\notin 2\pi\N_+$, where $\th=\Theta(g)$.
\begin{itemize}
\item[(1a)]
If $0\neq \th\in\La_-(S,x)$, then
$a_i=\ell_{\a_i}$ is a real-analytic function of
$\Hol(g) \in\Rep(\gru,\PSL_2(\R))$ in a neighbourhood of $g$.
\item[(1b)]
If $0\neq\th\in \La_0(S,x)$, then
$\displaystyle \frac{a_i}{a_j}=
\frac{\ell_{\a_i}}{\ell_{\a_j}}$ is a real-analytic function of
$\Hol(g) \in\Rep(\gru,\SE_2(\R))/\R_+$ in a neighbourhood of
$g\in\Theta^{-1}(\La_0)$.
\item[(2)]
If $\th=0$, then $\ti{a}_i=\tilde{\ell}_{\a_i}$
is a real-analytic function of $\Hol(g)\in\Rep(\gru,\PSL_2(\R))$
and $[\e]$ in a neighbourhood of $(g,[\e])\in\Theta^{-1}(0)$.
\end{itemize}
\end{proposition}

Because hyperbolic (resp. Euclidean)
triangles are characterized by the lengths of their
edges (resp. by the projectivization of the Euclidean
lengths of their edges),
it is thus clear that the holonomy together with an adapted
triangulation allow to reconstruct the full geometry of the surface.

\begin{corollary}\label{cor:adapted-coordinates}
Let $\ua$ be a triangulation on $(S,x)$.
\begin{itemize}
\item[(1a)]
If $\ua$ is adapted to $g\in\Yy(S,x)(\La_-)\setminus
\Theta^{-1}(0)$,
then $\ell_{\ua}$
is a local system of real-analytic coordinates on $\Yy(S,x)$ around $g$.
\item[(1b)]
If $\ua$ is adapted to $g\in\Yy(S,x)(\La_0)$, then
$\displaystyle\left\{
\frac{\ell_{\a_i}}{\ell_{\a_1}}\,\Big|\,i=2,\dots,N\right\}
\cup\{\chi(\dot{S},\th)\}$
is a local system of coordinates on $\Yy(S,x)$ around $g$.
\item[(2)]
If $\ua$ is adapted to $(g,[\e])\in\wh{\Theta}^{-1}(0)$, then
$\widetilde{\ell}_{\ua}$ is a local system of real-analytic coordinates
on $\Bl_0\Yy(S,x)$ around $(g,[\e])$.
\end{itemize}
\end{corollary}

The next task will be to produce at least one triangulation adapted
to $g$ for every $g\in\Bl_0\Yy(S,x)$.

\section{Voronoi decomposition}\label{sec:voronoi}

Let $(S,x)$ be a surface with a $\th$-admissible metric $g$.
For the moment, we assume $\Theta(g)\neq 0$,
so that the function $dist:\dot{S}\rar \R_{\geq 0}$
that measures the distance from $x$ is well-defined.

\begin{definition}
A {\bf shortest path} from $p\in\dot{S}$ is a (geodesic) path from $p$ to $x$
of length $dist(p)$.
\end{definition}

The concept of shortest path can extended to the whole $S$.
In fact, it is clear that at every $x_j$ with $\th_j>0$
the constant path is the only shortest one.

\begin{remark}
If $x_j$ is marks a cusp (resp. a cylinder),
then we can cure our definition as follows.
Consider a horoball $B_j$ around $x_j$ of small circumference
(resp. a semi-infinite cylinder $B_j$ ending at $x_j$),
so that no other conical points sit inside $B_j$
and all simple geodesics that enter $B_j$
end at $x_j$.
Let $\g$ be a nonconstant geodesic from $x_j$ to $x$,
which is made of two portions: $\g'$ from $x_j$ to the
first intersection point $y$ of $\g\cap B_j$ and $\g''=\g\setminus\g'$.
We say that $\g$ is
{\bf shortest} if $\ell_{\g''}=dist(y)$.
One can easily see that there are finitely many shortest paths
from a cusp (resp. a cylinder)
and that there is at least one (because $\pa B_j$ is compact).
\end{remark}

If $\th$ is small, then we can consider the modified distance
(with sign) $\widetilde{dist}: S\rar [-\infty,\infty]$
of a point in $S$ from $\pa B$, where $B$ is the standard
decoration and $\widetilde{dist}(p)$
is positive if and only if $p\in S\setminus B$.
Mimicking the trick as in the previous remark,
we can define a modified valence function $\wti{\val}$ on the whole $S$.
It is clear that $\val=\wti{\val}$.

Thus, we can define $\wti{d}$ and $\wti{\val}$ on a
projectively decorated surface $(S,x,[\e])$, by choosing
a system of small balls $B$ whose projectivized circumferences are $[\e]$.

\begin{definition}
The {\bf valence} $\val(p)$ of a point $p\in S$ is the number
of shortest paths at $p$. The {\bf Voronoi graph} $G(g)$ is the locus
of points of valence at least two.
\end{definition}

Because $g$ has constant curvature, one can conclude that $G(g)$
is a finite one-dimensional CW-complex embedded inside $\dot{S}$
with geodesic edges:
its vertices are $V(g)=\val^{-1}([3,\infty))$ and its (open) edges
are $E(g)=\pi_0(\val^{-1}(2))$.
Notice that the closure
$\ol{G(g)}$ passes through $x_j$ if and only if $\th_j=0$.

By definition, for every edge $e\in E(g)$ and for every $p\in e$,
there are exactly two shortest
paths $\ora{\b_1}(p)$ and $\ora{\b_2}(p)$ from $p$. Moreover,
the interior of $\ora{\b_i}(p)$ does not contain any other marked
point for $i=1,2$.
Then the composition $\a_e(p):=\ola{\b_1}(p)\ast\ora{\b_2}(p)$ is an arc
from some $x_j$ to some $x_j$
and its homotopy class (as arcs) $\a_e$ is independent of $p$.

\begin{remark}
The angle $\psi_0(e)$ at $x_j$ spanned by $\bigcup_{p\in e} \ola{\b_1}(p)$
is called ``edge invariant'' by Luo \cite{luo:functionals}.
\end{remark}

\begin{definition}
The (isotopy class of the) path $\a_e\subset S$ is the
{\bf arc dual to $e\in E(g)$} and $\ua(g)=\{\a_e\,|\, e\in E\}$
is the {\bf Voronoi system of arcs} for $g$.
\end{definition}

The complement $S\setminus\ua(g):=\bigcup_{v\in V} t_v$ is called
{\bf Voronoi decomposition}. The cell $t_v$ is
a pointed polygon if $v$ is a cusp and it is a polygon otherwise.

\begin{proposition}\label{prop:voronoi-adapted}
Let $g\in\Bl_0\Yy(S,x)$
be a hyperbolic/flat admissible metric (resp. a
hyperbolic admissible metric with a projective
decoration $[\e]$)
and let $\ua(g)$ its Voronoi system.
Consider a maximal system of arcs $\ua\supseteq\ua(g)$ such that
only one oriented arc in $\ua$ terminates at each cusp/cylinder
(resp. at each cusp $x_j$ with $\e_j=0$).
Then
\begin{itemize}
\item[(1)]
$\a_i$ is compatible with $g$;
\item[(2)]
the geodesic representative $\hat{\a}_i$ of each $\a_i\in\ua$ intersects
$x$ only at $\pa\hat{\a}_i$;
\item[(3)]
$\ua$ is adapted to $g$.
\end{itemize}
\end{proposition}

\begin{proof}
We only deal with the case $\Theta\neq 0$. The decorated case
is similar and so we omit the details.

Suppose that $\hat{\a}_i$ joins $x_j$ to $x_k$ (possibly $j=k$).
Let $e$ be the edge of the Voronoi graph $G(g)$
dual to $\a_i$ (which may reduce to a vertex) and call $v_0$
the point of $e$ which is closest to $x_j$ and $x_k$.
Let $\ora{\b_j}(v_0)$ (resp. $\ora{\b_k}(v_0)$) be the shortest
path from $v_0$ to $x_j$ (resp. $x_k$), so that
$\a_i\simeq \ola{\b_j}(v_0)\ast\ora{\b_k}(v_0)$.
%

\begin{center}
\begin{figurehere}
\psfrag{xj}{$x_j$}
\psfrag{xk}{$x_k$}
\psfrag{v0}{$v_0$}
\psfrag{vt}{$v_t$}
\psfrag{v1}{$v_1$}
\psfrag{bj(v0)}{$\ora{\b_j}(v_0)$}
\psfrag{bj(vt)}{$\ora{\b_j}(v_t)$}
\psfrag{bk(v1)}{$\ora{\b_k}(v_1)$}
\psfrag{e}{$e$}
\psfrag{d}{$\delta(v_t)$}
\psfrag{e'}{$e'$}
\includegraphics[width=0.7\textwidth]{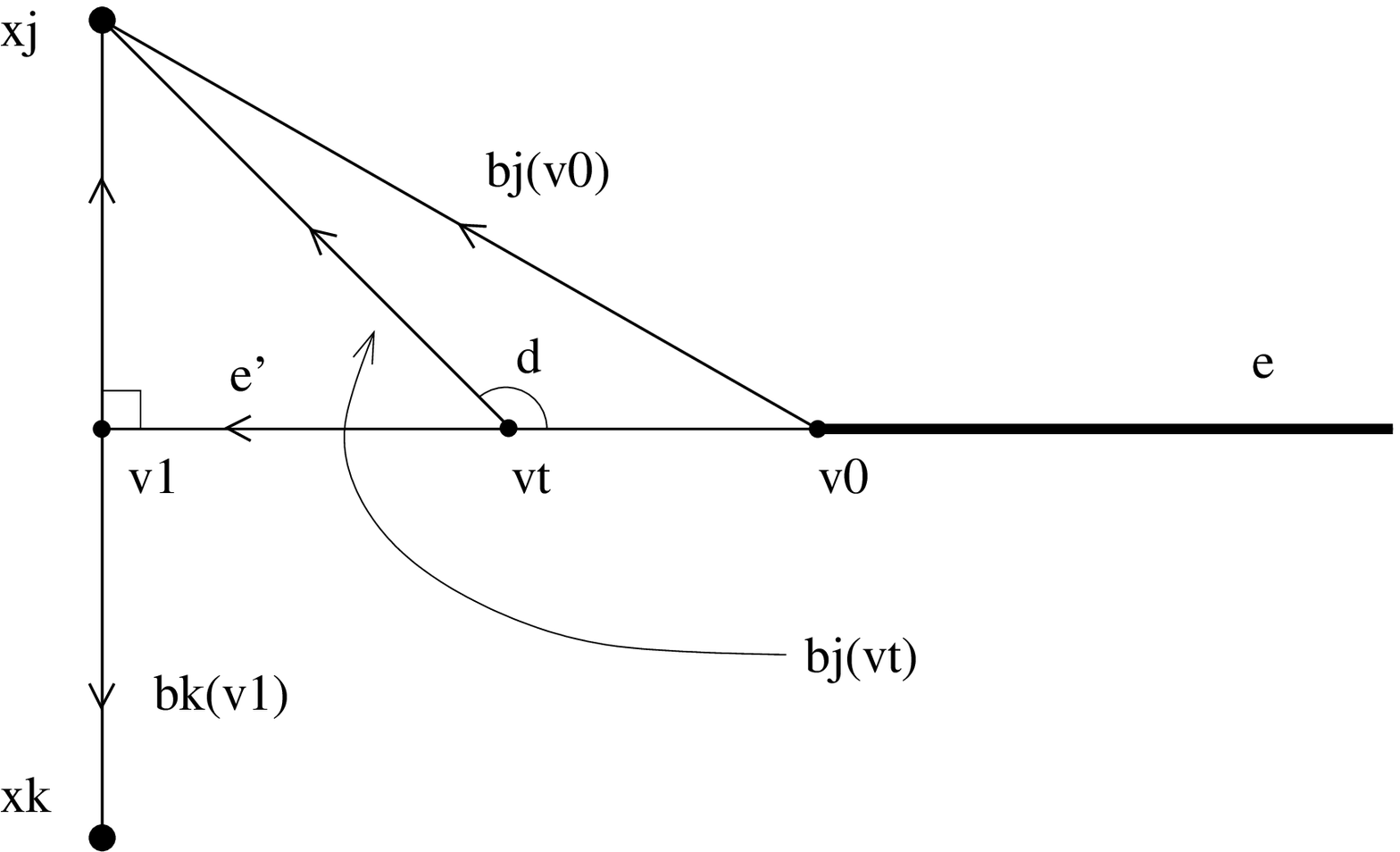}
\caption{The case in which $e'\neq\{v_0\}$.}
\end{figurehere}
\end{center}

Consider the maximal closed geodesic segment $e'$
that starts at $v_0$
and such that,
for every $v\in e'$, the shortest path $\ora{\b_j}(v)$
from $v$ to $x_j$ homotopic to $\ora{vv_0}\ast\ora{\b_j}(v_0)$
and the shortest path $\ora{\b_k}(v)$
from $v$ to $x_k$ homotopic to $\ora{vv_0}\ast\ora{\b_k}(v_0)$
satisfy $\ell(\b_j(v))=\ell(\b_k(v))\leq\ell(\b_j(v_0))=\ell(\b_k(v_0))$.
Call $\delta(v)$ the angle $\wh{v_0 v \b_j}=\wh{v_0 v \b_k}$.

If $e'=\{v_0\}\subset e$, then $\delta(v_0)=\pi/2$ and
$\mathrm{int}(\b_j(v_0))\cap x=
\mathrm{int}(\b_k(v_0))\cap x=\emptyset$;
so $\ola{\b_j}(v_0)\ast\ora{\b_k}(v_0)$
is already the desired smooth geodesic $\hat{\a}_i$.

Otherwise, start travelling along $e'$ from $v_0$ until the point $v_1$
which is closest to $x_j$ and $x_k$. Call $v_t$ the points of $e'$
between $v_0$ and $v_1$ for $t\in(0,1)$.
Clearly, $\delta(v_1)=\pi/2$
and $\delta(v_t)$ is a strictly decreasing
function of $t$.

As a consequence, $d(v_0,y)<d(v_0,x_j)$
for all $y\in\mathrm{int}(\b_j(v_t))$ and $t\in(0,1]$
(and similarly for $x_k$). Thus, $\mathrm{int}(\b_j(v_t))\cap x=
\mathrm{int}(\b_k(v_t))\cap x=\emptyset$ for $t\in[0,1]$.

We can conclude that $\a_i(t):=\ola{\b_j}(v_t)\ast \ora{\b_k}(v_t)$
is the wished homotopy of arcs between $\a_i\simeq \a_i(0)$ and
the smooth geodesic $\hat{\a}_i:=\a_i(1)$.

Parts (2) and (3) clearly follow from (1).
\end{proof}

\begin{remark}
It was shown by Rivin \cite{rivin:euclidean} (in the flat case)
and by Leibon \cite{leibon:hyperbolic} (in the hyperbolic case)
that the Voronoi construction gives a $\Mod(S,x)$-equivariant
cellularization of $\Yy(S,x)$: the affine coordinates on each cell
are given by $\{\psi_0(e)\,|\,e\in E(g)\}$ (Luo \cite{luo:rigidity}
has shown that one can also use different curvature functions $\psi_k$).
This is similar to what happens for surfaces with geodesic boundary,
after replacing $\psi_0$ by the analogous quantity
\cite{luo:boundary} \cite{mondello:WP}. However, the cone parameters
$\psi_0(e)$ must obey some extra constraints, because the sum of
the internal angles of a triangle $t$ cannot exceed $\pi$.
Thus, the cells of $\Yy(S,x)$ are {\it truncated} simplices.
\end{remark}

%
\section{An explicit formula}\label{sec:explicit}

Similarly to \cite{penner:volumes} and \cite{mondello:WP},
we want now to provide an explicit formula for $\eta$
in terms of the $a$-lengths, using techniques from \cite{goldman:hamiltonian}.

\begin{theorem}\label{thm:poisson}
Let $\ua$ be a triangulation of $(S,x)$ adapted to
$g\in\Yy(S,x)(\La_-^\circ)$ and let $a_k=\ell_{\a_k}$.
Then the Poisson structure $\eta$ at $g$ can be expressed
in terms of the $a$-lengths as follows
\[
\eta_g=
\sum_{h=1}^n
\sum_{\substack{\mathrm{s}(\ora{\a_i})=x_h\\
\mathrm{s}(\ora{\a_j})=x_h}}
\frac{\sin(\th_h/2-d(\ora{\a_i},\ora{\a_j}))}{\sin(\th_h/2)}
\frac{\pa}{\pa a_i}\wedge
\frac{\pa}{\pa a_j}
\]
%
%
where $s(\ora{\a}_k)$ is the starting point of the oriented
geodesic arc $\ora{\a}_k$ and $d(\ora{\a_i},\ora{\a_j})$ is the angle
spanned by rotating the tangent vector to $\ora{\a_i}$
at its starting point
clockwise to the tangent vector at the starting point of
$\ora{\a_j}$.
If $\th\in(0,2\pi)^n$, then the above formula also expresses
$8$-times the Weil-Petersson dual symplectic form $\eta_{WP,\th}$
at $g\in \Teich(S,x)$.
\end{theorem}

\begin{remark}
In \cite{mondello:WP} a similar formula for hyperbolic surfaces with
geodesic boundary is proven. Really, if $\Sigma$ is a surface with boundary,
and $d\Sigma$ is its double with the natural real involution $\s$,
then $\pi_{\i}:\Teich(d\Sigma)^{\s}\rar\Teich(\Sigma)$ has the property
that $(\pi_\i)_*\eta_{WP,dS}=2\eta_{WP,S}$,
and not $\eta_{WP,S}$, as claimed in Proposition~{1.7} of \cite{mondello:WP}.
This explains why the two formulae
are off by a factor $2$.
\end{remark}

\begin{proof}[Proof of Theorem~\ref{thm:poisson}]
We want to compute $\eta_g(da_i,da_j)$.
Fix a basepoint $p\in \dot{S}$ and call $\g(\ora{\a_k})$ the parabolic
element of $\gru:=\pi_1(\dot{S},p)$ that winds around $s(\ora{\a_k})$,
in such a way that $\g(\ora{\a_k})\ast \g(\ola{\a_k})$ corresponds
to the arc $\a_k$.

Let $\rho:=\Hol(g)$ and let $u\in H^1(\dot{S};\xi)$ be a tangent
vector in $T_{\rho}\Rep(\gru,\PSL_2(\R))$.
The deformation of $\rho$ corresponding to $u$ can be written as
$\rho_t(\g)=\rho(\g)+t u(\g)\rho(\g)+O(t^2)$ and we will
also write
$S_k(t)=\rho_t(\g(\ora{\a_k}))$ and $s_k=\log(S_k)$,
and similarly
$F_k(t)=\rho_t(\g(\ola{\a_k}))$ and $f_k=\log(F_k)$.

Because of Lemma~\ref{lemma:technical}(c),
\[
\Bb(da_i,da_j)=
\frac{4 \Bb(d\Bb(s_i,f_i)\cap d\Bb(s_j,f_j))}
{\sinh(a_i)\sinh(a_j)\th_{s(\ora{\a_i})}
\th_{s(\ola{\a_i})}\th_{s(\ora{\a_j})}
\th_{s(\ola{\a_j})}}
\]
The numerator potentially contains $4$ summands: we will only compute
the one occurring when $s(\ora{\a_i})=s(\ora{\a_j})$, as the others
will be similar. In particular, because of Lemma~\ref{lemma:technical}(b),
we need to calculate
$\Bb(R_i\otimes \g(\ora{\a_i})\cap R_j\otimes \g(\ora{\a_j}))$,
where 
$R_k:=(1-\Ad_{S_k}^{-1})^{-1}[f_k,s_k]$,
because $s_k\otimes \g(\ora{\a_k})$ (resp. $f_k\otimes\g(\ola{\a_k})$)
is a multiple of $d\th_{s(\ora{\a_k})}$ (resp.
$d\th_{s(\ola{\a_k})}$) by Lemma~\ref{lemma:technical}(a)
and $d\th_h$ belongs to the radical of $\eta$ for every $h$.
%

The local situation around $s(\ora{\a_i})$ is described
in Figure~\ref{fig:eta}.

\begin{center}
\begin{figurehere}
\psfrag{fi}{$s(\ola{\a_i})$}
\psfrag{fj}{$s(\ola{\a_j})$}
\psfrag{si}{$s(\ora{\a_i})$}
\psfrag{p}{$p$}
\psfrag{p'}{$p'$}
\psfrag{gfi}{$\g(\ola{\a_i})$}
\psfrag{gfj}{$\g(\ola{\a_j})$}
\psfrag{gsi}{$\g(\ora{\a_i})$}
\psfrag{gsj}{$\g(\ora{\a_j})$}
\psfrag{y1}{$y_1$}
\psfrag{y2}{$y_2$}
\includegraphics[width=0.7\textwidth]{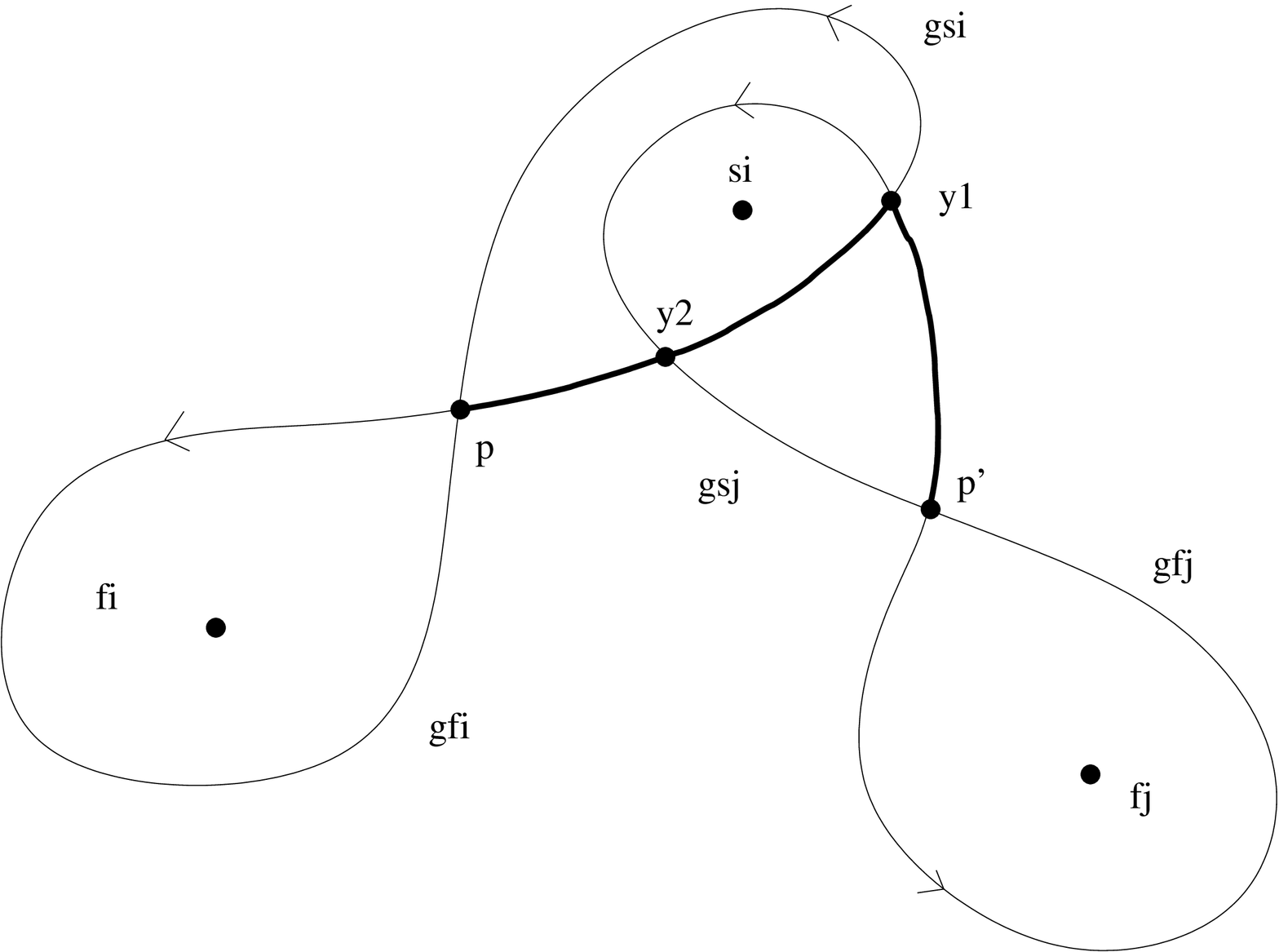}
\caption{The bundle $\xi$ is trivialized along the thick
path.}\label{fig:eta}
\end{figurehere}
\end{center}

The intersection pairing at the level of $1$-chains gives
$\g(\ora{\a_i})\cap \g(\ora{\a_j})=y_1-y_2$.
Because we have trivialized $\xi$ on the thick part,
we obtain
\[
\Bb(R_i\otimes \g(\ora{\a_i})\cap R_j\otimes \g(\ora{\a_j}))
= \Bb(R_i,(1-\Ad_{S_j}^{-1}) R_j) =
\Bb(R_i,[f_j,s_j])
\]
By Lemma~\ref{lemma:trig},
\[
[s_k,f_k]=\frac{1}{4}\th_{s(\ora{\a_k})}\th_{s(\ola{\a_k})}[L(S_k),L(F_k)]
=\frac{1}{2}\th_{s(\ora{\a_k})}\th_{s(\ola{\a_k})}\sinh(a_k) L(\ora{\a_k})
\]
where $L(\ora{\a_k})$ is the axis of the geodesic $\ora{\hat{\a}_k}$.

So far we have obtained
\begin{align*}
\Bb( (1-\Ad_{S_i^{-1}})^{-1}[f_i,s_i],[f_j,s_j]) =
\frac{1}{4}\th_{s(\ora{\a_i})}\th_{s(\ola{\a_i})}
& \th_{s(\ora{\a_j})}\th_{s(\ola{\a_j})}
\sinh(a_i)\sinh(a_j)\cdot \\
& \cdot \Bb( (1-\Ad_{S_i}^{-1})^{-1} L(\ora{\a_i}),L(\ora{\a_j}))
\end{align*}
Notice that $\Ad_{S_i^h}=\exp (h\,\ad_{s_i})$
acts on $L(\ora{\a_i})$ as
a rotation of angle $h\nu$ centered at $s(\ora{\a_i})$,
where $\nu=\th_{s(\ora{\a_i})}$,
and so
\[
\Bb(\Ad_{S_i^h}L(\ora{\a_i}),L(\ora{\a_j}))=
2\cos(-\d+h\nu)=
2\mathrm{Re}\left[\exp((-\d)\sqrt{-1}+h\nu\sqrt{-1})
\right]
\]
where $\d=d(\ora{\a_i},\ora{\a_j})$.
Hence,
\[
\Bb(w(\ad_{s_i})L(\ora{\a_i}),L(\ora{\a_j}))=2
\mathrm{Re}\left[\exp( -\d \sqrt{-1} ) w( \nu\sqrt{-1}) \right]
\]
where $w$ is an analytic function.

Therefore, we can conclude that
\[
\Bb(R_i\otimes \g(\ora{\a_i})\cap R_j\otimes \g(\ora{\a_j}))
= \frac{1}{4}\th_{s(\ora{\a_i})}\th_{s(\ola{\a_i})}
\th_{s(\ora{\a_j})}\th_{s(\ola{\a_j})}\sinh(a_i)\sinh(a_j)
\frac{\sin(\th_{s(\ora{\a_i})} /2-\d)}{\sin(\th_{s(\ora{\a_i})}/2)}
\]
because $\displaystyle
2\mathrm{Re}\left[
\frac{\exp(-\d\sqrt{-1})}{1-\exp(-\nu\sqrt{-1})} \right]
=\frac{\sin (\nu/2-\d)}{\sin(\nu/2)}$.

Finally, the first summand of $\Bb(da_i,da_j)$ is
$\displaystyle
\frac{\sin (\th_{s(\ora{\a_i})} /2-d(\ora{\a_i},\ora{\a_j}))}
{\sin(\th_{s(\ora{\a_i})} /2)}.$
\end{proof}

To complete the proof of the theorem, we only need to establish the
following.

\begin{lemma}\label{lemma:technical}
\[
\begin{array}{ll}
{\text{\rm{(a)}}} &
d\th_{s(\ora{\a_k})}=L(S_k)\otimes \g(\ora{\a_k})\\
{\text{\rm{(b)}}} &
d\Bb(s_k(t),f_k(t))
= (1-\Ad_{F_k^{-1}})^{-1}[s_k,f_k]
\otimes \g(\ola{\a_k})
+(1-\Ad_{S_k^{-1}})^{-1}[f_k,s_k]
\otimes\g(\ora{\a_k})+\\
& \qquad\qquad\qquad\qquad
+\frac{\Bb(f_k,f_k)}{\Bb(s_k,f_k)}
f_k\otimes\g(\ola{\a_k})+
\frac{\Bb(s_k,s_k)}{\Bb(f_k,s_k)}
s_k\otimes\g(\ora{\a_k})
\\
{\text{\rm{(c)}}} &
\displaystyle
\sinh(a_k)da_k=
\left[
\frac{2d\th_{s(\ora{\a_k})}}{\th_{s(\ora{\a_k})}^2\th_{s(\ola{\a_k})}}
+\frac{2d\th_{s(\ola{\a_k})}}{\th_{s(\ola{\a_k})}^2\th_{s(\ora{\a_k})}}
\right]\Bb(s_k,f_k)
-
\frac{2 d\Bb(s_k,f_k)}{\th_{s(\ora{\a_k})}
\th_{s(\ola{\a_k})} }
\end{array}
\]
as elements of $T^*_g \Yy(S,x)\cong H_1(\dot{S};\xi)$.
\end{lemma}

\begin{proof}
Part (a) was essentially proved in \cite{goldman:hamiltonian} and
part (c) is easily obtained from Lemma~\ref{lemma:trig}(a) by
differentiation.

For part (b), consider the function $\Bb(s_k(t),f_k(t))$
along the path $t\mapsto \rho_t=\exp(tu)\rho=\rho+tu\rho+O(t^2)$,
where $s_k(0)=s_k$ and $f_k(0)=f_k$.
By Lemma~\ref{lemma:log}
\[
s_k(t)=\log\left[\exp(tu_{\ora{k}} ))\exp(s_k) \right]=
s_k+t (1-\Ad_{S_k})^{-1}
[s_k,u_{\ora{k}}]+t\frac{\Bb(u_{\ora{k}},s_k)}{\Bb(s_k,s_k)}+O(t^2)
\]
where $u_{\ora{k}}=u(\g(\ora{\a_k}))$ and $u_{\ola{k}}=
u(\g(\ola{\a_k}))$.
Hence,
\begin{align*}
\Bb(s_k(t),f_k(t))
& =\Bb(s_k,f_k)
+t\Bb(s_k,(1-\Ad_{F_k})^{-1}[f_k,u_{\ola{k}}])+
t\frac{\Bb(u_{\ola{k}},f_k)}
{\Bb(f_k,f_k)}\Bb(s_k,f_k)+ \\
& \qquad\qquad +
t\Bb(f_k,(1-\Ad_{S_k})^{-1}[s_k,u_{\ora{k}}])+t\frac{\Bb(u_{\ora{k}},s_k)}
{\Bb(s_k,s_k)}\Bb(f_k,s_k)+O(t^2)= \\
& = \Bb(s_k,f_k)+t\Bb(u_{\ola{k}},(1-\Ad_{F_k^{-1}})^{-1}[s_k,f_k])+
t\frac{\Bb(f_k,f_k)}{\Bb(s_k,f_k)}
\Bb(u_{\ola{k}},f_k)
+ \\
& \qquad\qquad + t\Bb(u_{\ora{k}},(1-\Ad_{S_k^{-1}})^{-1}[f_k,s_k])+
t
\frac{\Bb(s_k,s_k)}{\Bb(f_k,s_k)}
\Bb(u_{\ora{k}},s_k)+O(t^2)
\end{align*}
Finally,
\begin{align*}
d\Bb(s_k(t),f_k(t))
& = (1-\Ad_{F_k^{-1}})^{-1}[s_k,f_k]
\otimes \g(\ola{\a_k})
+(1-\Ad_{S_k^{-1}})^{-1}[f_k,s_k]
\otimes\g(\ora{\a_k})+ \\
& \qquad\qquad +\frac{\Bb(f_k,f_k)}{\Bb(s_k,f_k)}
f_k\otimes\g(\ola{\a_k})+
\frac{\Bb(s_k,s_k)}{\Bb(f_k,s_k)}
s_k\otimes\g(\ora{\a_k})
\end{align*}
\end{proof}

%
\appendix
\section{Some linear algebra}

Let $R\in\PSL_2(\R)$ be a hyperbolic element
corresponding to the oriented geodesic $\ora{\b}$ in $\H$.
Define $L(R)=2r/\ell(R)\in\psl_2(\R)$,
where $r=\log(R)$ is the unique logarithm of $R$ in $\psl_2(\R)$
and $\ell(R)=\mathrm{arccosh}(\tr(R^2)/2)$
is the translation distance of $R$, so that $\Bb(L(R),L(R))=2$.

\begin{remark}
Given an oriented hyperbolic geodesic $\ora{\b}$ in $\H$,
we say that a component of $\H\setminus\b$ is
the {\it $\b$-positive half-plane} if it induces the orientation
of $\ora{\b}$ on its boundary.
The definition of positive half-plane with respect to an
oriented line in $\R^2$ is similar.
\end{remark}

If $S\in\PSL_2(\R)$ is elliptic of angle $\nu=\mathrm{arccos}(\tr(S^2)/2)$,
then
define $L(S)=2s/\nu\in\psl_2(\R)$, where $s=\log(S)$
is an infinitesimal counterclockwise rotation, so that
$\Bb(L(S),L(S))=-2$.

Simple considerations of hyperbolic geometry give the following
(see \cite{ratcliffe:foundations}, for instance).

\begin{lemma}\label{lemma:trig}
{\rm{(a)}}
Let $S_1,S_2\in\PSL_2(\R)$ be elliptic elements that fix
distinct points $x_1,x_2\in\H$ and let $R$ be the hyperbolic element
that fixes the unique geodesic through $x_1$ and $x_2$ and takes
$x_1$ to $x_2$. Then
\begin{align*}
\Bb(L(S_1),L(S_2)) & = -2\cosh(d(x_1,x_2)) \\
[L(S_1),L(S_2)] & =
2\sinh(d(x_1,x_2)) L(R)
\end{align*}
where $d(x_1,x_2)$ is the hyperbolic distance between $x_1$ and $x_2$.

{\rm{(b)}}
Let $R_1,R_2\in\PSL_2(\R)$ be hyperbolic elements corresponding
to oriented geodesics $\ora{\b_1},\ora{\b_2}$ on $\H$. Then
\[
\Bb(L(R_1),L(R_2))=
\begin{cases}
2\cos(\d) & \text{if they meet forming an angle $\d$} \\
2\cosh(d(\b_1,\b_2)) & \text{if they are disjoint.}
\end{cases}
\]

{\rm{(c)}}
Let $R\in\PSL_2(\R)$ be a hyperbolic element corresponding to $\ora{\b}$
and $S\in\PSL_2(\R)$ be an elliptic element that fixes $x\in\H$.
Then
\[
\Bb(L(R),L(S))=-2 \sinh(d(\ora{\b},x))
\]
where $d(\ora{\b},x)$ is positive if $x$ lies in the $\ora{\b}$-positive
half-plane.
\end{lemma}

In the flat case, we will only need the following simple result.

\begin{lemma}\label{lemma:trig-flat}
{\rm{(a)}}
Let $S_1,S_2\in\mathrm{SE}_2(\R)$ be elliptic elements,
namely $S_i(v)=N_i(v)+w_i$ with $1\neq N_i\in\mathrm{SO}_2(\R)$
and $w_i\in\R^2$ for $i=1,2$. Thus, $S_i$ has a fixed point
$x_i=(1-N_i)^{-1}w_i$
and the Euclidean distance $d(x_1,x_2)$ can be expressed as
\[
d(x_1,x_2)=\left\|(1-N_1)^{-1}w_1-(1-N_2)^{-1}w_2 \right\|
\]

{\rm{(b)}}
Given elliptic elements $S_1,S_2,S_3\in\mathrm{SE}_2(\R)$
with fixed points $x_1,x_2,x_3$, then the quantity
\[
x_1\wedge x_2 + x_2\wedge x_3+x_3\wedge x_1
\in\Lambda^2\R^2\cong\R
\]
is
positive (resp. negative, or zero) if and only if
$x_3$ lies in the positive half-plane with respect to the line
determined by $\ora{x_1 x_2}$ (resp. the negative half-plane,
or the three points are collinear).
\end{lemma}

Finally, the following explicit expression is needed
in the proof of Lemma~\ref{lemma:technical}.

\begin{lemma}\label{lemma:log}
Let $s,u\in\psl_2(\R)$ such that $s$ is elliptic or
hyperbolic and let $S=\exp(s)$. Then
\[
\log(\exp(tu)S)=s+t
(1-\Ad_S)^{-1}[u,s]
+t\frac{\Bb(u,s)}{\Bb(s,s)}s
+O(t^2)
\]
where $(1-\Ad_S)$ is here interpreted as
an automorphism of $s^{\perp}\subset\psl_2(\R)$.
\end{lemma}

\begin{proof}
Extend $\Bb$ to $\mathfrak{gl}_2(\R)$, so that $\Bb(x,y)=\tr(xy)$
for $x,y\in\mathfrak{gl}_2(\R)$, and consider $(1-\Ad_S)
\in\mathrm{End}(\mathfrak{gl}_2(\R))$.

Because $s$ is elliptic or hyperbolic, then $\displaystyle s^2=
\left(\begin{array}{cc}c & 0 \\ 0 & c \end{array} \right)$
with $c\neq 0$, and so
$\Bb(s,s)\neq 0$.
Hence, $V:=\mathrm{ker}(1-\Ad_S)=
\mathrm{span}\{1,s\}$ and $\mathfrak{gl}_2(\R)=V\oplus W$ is
an orthogonal decomposition, where
$W=\mathrm{Im}(1-\Ad_S)$.

Notice also that multiplying by $s$
(and so by $S$ or $S^{-1}$) on the left
or on the right is an automorphism of $\mathfrak{gl}_2(\R)$
that preserves $V$ and $W$.
Define $M_S:\mathfrak{gl}_2(\R)\lra \mathfrak{gl}_2(\R)$ as
\[
M_S(x+y):=(1-\Ad_{S})\Big|_W^{-1}(x)
\qquad \text{where $x\in W$ and $y\in V$}
\]
Clearly, the multiplication by $s$ (or by $S$ or $S^{-1}$)
commutes with $\Ad_S$, and so also with $M_S$.

Because the first-order term in $t$
in the equality we want to prove is also linear in $u$,
it is sufficient to
compute the exponential $E$ of the right hand side (up to $O(t^2)$)
in two different cases: $u=s$ and $u\in W$,
since $\psl_2(\R)=W\oplus \R s$.

For $u=s$, we have $[u,s]=0$ and so
\begin{align*}
E & = \exp\left(s+t\frac{\Bb(s,s)}{\Bb(s,s)}s\right) = \exp(s+ts) = \\
& = S \exp(ts)=S (1+ts+O(t^2))=S+ts S+O(t^2)
\end{align*}

If $u\in W$, then $(1-\Ad_S)(u),(1-\Ad_S)(uS)\in W$.
Hence,
\begin{align*}
E & = \exp(s+
t(1-\Ad_S)^{-1}[u,s])= \\
& = S+  t
\sum_{h\geq 1}\frac{1}{h!}
\sum_{j=0}^{h-1}
s^j
M_S^{-1}([u,s]) s^{h-1-j} +O(t^2)= \\
& = S +  t
\sum_{h\geq 1}\frac{1}{h!}
\sum_{j=0}^{h-1}
M_S^{-1}(s^j [u,s] s^{h-1-j}) +O(t^2)= \\
& = S +  t
\sum_{h\geq 1}
M_S^{-1}([u,s^h/h!]) +O(t^2)= \\
& = S +  t
M_S^{-1}(uS-Su) +O(t^2)= S+t M_S^{-1}(1-\Ad_S)(uS) +O(t^2)=\\
& = S+tuS+O(t^2).
\end{align*}
\end{proof}


\bibliographystyle{amsalpha}
\bibliography{AB3}

\end{document}